\documentclass[12pt]{amsart}
\usepackage[latin1]{inputenc}
\usepackage{amssymb, amsmath, amscd, mathrsfs,marvosym, psfrag,color,a4} 

\input xy
\xyoption{all}
\DeclareMathAlphabet{\mathpzc}{OT1}{pzc}{m}{it}

\newtheorem{theorem}{Theorem}[section]

\newtheorem{proposition}[theorem]{Proposition}
\newtheorem{corollary}[theorem]{Corollary}

\newtheorem{lemma}[theorem]{Lemma}

\theoremstyle{definition}
\newtheorem{definition}[theorem]{Definition}

\theoremstyle{remark}
\newtheorem{remark}[theorem]{Remark}
\newtheorem{remarks}[theorem]{Remarks}

\def\le{\leqslant}

\newcommand{\CA}{{\mathcal A}}

\newcommand{\CF}{{\mathcal F}}
\newcommand{\CG}{{\mathcal G}}

\newcommand{\CI}{{\mathcal I}}
\newcommand{\CJ}{{\mathcal J}}
\newcommand{\CK}{{\mathcal K}}

\newcommand{\CO}{{\mathcal O}}

\newcommand{\CS}{{\mathcal S}}
\newcommand{\CT}{{\mathcal T}}

\newcommand{\CW}{{\mathcal W}}

\newcommand{\CX}{{\mathcal X}}

\newcommand{\CZ}{{\mathcal Z}}

\newcommand{\SA}{{\mathscr A}}
\newcommand{\SB}{{\mathscr B}}
\newcommand{\SC}{{\mathscr C}}

\newcommand{\SM}{{\mathscr M}}
\newcommand{\SN}{{\mathscr N}}

\newcommand{\SQ}{{\mathscr Q}}

\newcommand{\SV}{{\mathscr V}}

\newcommand{\hCW}{{\widehat\CW}}

\newcommand{\hCS}{{\widehat\CS}}

\newcommand{\DC}{{\mathbb C}}

\newcommand{\DR}{{\mathbb R}}
\newcommand{\DZ}{{\mathbb Z}}

\newcommand{\bA}{{\mathbf A}}
\newcommand{\bB}{{\mathbf B}}
\newcommand{\bC}{{\mathbf C}}

\newcommand{\bP}{{\mathbf P}}

\newcommand{\ch}{{\operatorname{char}\, }}

\newcommand{\Hom}{{\operatorname{Hom}}}

\newcommand{\supp}{{\operatorname{supp}}}
\newcommand{\catmod}{{\operatorname{-mod}}}

\newcommand{\im}{{\operatorname{im}\,}}

\newcommand{\ul}{\underline}

\newcommand{\pr}{{\operatorname{pr}}}

\newcommand{\id}{{\operatorname{id}}}

\newcommand{\sur}{\mbox{$\to\!\!\!\!\!\to$}}

\newcommand{\comment}[1]{}

\begin{document}

\pagenumbering{arabic}
\title[]{ Sheaves on the  alcoves I: Projectivity and wall crossing functors} 

\author[]{Peter Fiebig, Martina Lanini}
\begin{abstract} This paper is the first in a series of papers in which we define and study a category  of ``sheaves of $\CZ$-modules on the set of alcoves'' that carries important information on the category of representations of semisimple Lie algebras in positive characteristics. Here, we define this category and study the structure of its projective objects using a new version of wall crossing functors.  
\end{abstract}

\address{Department Mathematik, FAU Erlangen--N\"urnberg, Cauerstra\ss e 11, 91058 Erlangen}
\email{fiebig@math.fau.de}
\address{Universit\`a degli Studi di Roma ``Tor Vergata", Dipartimento di Matematica, Via della Ricerca Scientifica 1, I-00133 Rome, Italy }
\email{lanini@mat.uniroma2.it}
\maketitle
\section{Introduction}

Let $R$ be an irreducible root system in a real vector space $V$, and denote by  $\CA$ the corresponding set of alcoves, i.e. the set of connected components of the complement  in $V$ of the union of affine reflection hyperplanes. The choice of a system of positive roots endows $\CA$ with a natural partial order, known as the {\em generic Bruhat order}. Now we  view $\CA$ as  a topological space with the order ideals as open sets. 

We fix a field $k$ and denote by $S$ the symmetric algebra over $k$ associated with  the coweight lattice.  There is a certain commutative $S$-algebra $\CZ$, depending on the root system, that appears in surprisingly many contexts. For $k=\DC$, it is the deformed categorical center of the principal block of the category $\CO$ of the finite dimensional simple complex Lie algebra associated with $R$. 
It is also closely related to the categorical center of the principal block of restricted critical level category $\CO$ of the corresponding affine Kac--Moody algebra.
For any $k$, it is the torus equivariant cohomology with coefficients in $k$ of the flag variety $G^\vee/B^\vee$, where $G^\vee$ is the semisimple simply connected complex algebraic group for the dual root system, and $B^\vee\subset G^\vee$ is a Borel subgroup. For arbitrary $k$, $\CZ$ is the structure algebra of the Bruhat graph associated with the root system $R$.

In this paper we study {\em sheaves of $\CZ$-modules} on the topological space $\CA$, and we introduce a particular full subcategory $\bC$ of  sheaves that satisfy some natural conditions. Apart from some minor finiteness and reflexivity assumptions, the two important conditions are what we call the {\em support condition} and the {\em local extension condition}. The first concerns the structure of the $\CZ$-modules of sections supported on  locally  closed subsets, and the second the extension structure along root strings in $\CA$. For example, the standard objects in $\bC$ are certain skyscraper sheaves $\SV(A)$ supported on a single alcove $A$.  

The category $\bC$ inherits an exact structure from the surrounding category of sheaves of $\CZ$-modules.  For any alcove we show that there is an up to isomorphism unique projective object $\SB(A)$ in $\bC$ that is indecomposable and admits an epimorphism onto $\SV(A)$ for any alcove $A$.  

This paper is the first of a series of papers that is devoted to the study of the category $\bC$ and, in particular, its objects $\SB(A)$. We show in forthcoming papers that they  carry significant representation theoretic information. For example, the ranks of the stalks of the $\SB(A)$ yield, if the characteristic of $k$ is larger than the Coxeter number, the Jordan--H\"older multiplicities of baby Verma modules for the $k$-Lie algebra associated with $R$, and hence they encode the simple characters of the corresponding semisimple algebraic groups.

\subsection*{Contents}
In Section \ref{sec-alcoves} we recall the definitions of the set of alcoves and of the generic Bruhat order.  Some results concerning the right action of simple affine reflections on $\CA$ are provided. 
The algebra $\CZ$ is defined in Section \ref{sec-strucalg}, where we also study the structure of $\CZ$-modules  and introduce the notion of the $\CZ$-support. 

Section \ref{sec-sheavesparord} states some basic results on (pre-)sheaves on partially ordered sets with values in an abelian category. 
Some sufficient conditions for a presheaf to be a sheaf are given.  
We define subquotients of flabby presheaves, and study  short exact sequences.

In Section \ref{sec-sheavesA} we study sheaves of modules over the structure algebra $\CZ$  on the partially ordered set $\CA$. We discuss two important properties of such sheaves, the {\em support condition} and the {\em local extension condition}, which essentially define the category $\bC$. It contains a subcategory $\bB$ of objects that {\em admit a Verma flag}.  In Section \ref{sec-pfspec} we associate a sheaf to any locally closed subset of $\CA$. As a particular example of this we obtain  the {\em standard objects} $\SV(A)$. 

The category $\bC$ is not abelian, but it inherits an exact structure from its surrounding category of sheaves of $\CZ$-modules on $\CA$.  In Section \ref{sec-projobj} we study projective objects in $\bC$. Proposition \ref{prop-projobj} gives a sufficient condition for an object to be projective. This result is  then used to show that the sheaves associated to special sections are projective. These are the only projectives that we can construct directly, and they serve as the starting points for a wall crossing algorithm that can be used to prove the existence (and some properties) of the other indecomposable projectives. 

In the final Section \ref{sec-wcf} the wall crossing functors $\vartheta_s\colon\bC\to\bC$ for affine simple reflections $s$  are constructed. We describe the wall crossing algorithm, which also implies that the projectives actually admit a Verma flag, i.e. are contained in the subcategory $\bB$.

\subsection*{Acknowledgements}
This material is based upon work supported by the National Science Foundation under Grant No. 0932078 000 while the first author was in residence at the Mathematical Sciences Research Institute in Berkeley, California, during the Fall 2014 semester. The second author would like to thank the University of Edinburgh, that supported her research during the final part of this project.
 Both authors were partially supported by the DFG grant SP1388.

\section{Alcove Geometry}\label{sec-alcoves}

Fix  a finite irreducible root system $R$ in a real finite dimensional vector space  $V$.  For any $\alpha\in R$  denote by $\alpha^\vee\in V^\ast=\Hom_\DR(V,\DR)$ the corresponding coroot. Let $\langle\cdot,\cdot\rangle\colon V\times V^\ast\to \DR$ be the natural pairing. 
Let
\begin{align*}
X&:=\{\lambda\in V\mid \langle \lambda,\alpha^\vee\rangle\in\DZ\text{ for all $\alpha\in R$}\},\\
X^\vee&:=\{v\in V^\ast\mid \langle\alpha,v\rangle\in\DZ\text{ for all $\alpha\in R$}\},
\end{align*}
be the weight and the coweight lattice, resp.
Note that the root lattice $\DZ R$ is contained in $X$. 
Let $R^+\subset R$ be a system of positive roots. The dominant Weyl chamber is 
$$
C^+:=\{\lambda\in V\mid \langle\lambda,\alpha^\vee\rangle>0\text{ for all $\alpha\in R^+$}\},
$$
and the antidominant Weyl chamber is 
$$
C^-:=\{\lambda\in V\mid \langle\lambda,\alpha^\vee\rangle<0\text{ for all $\alpha\in R^+$}\}.
$$

\subsection{Alcoves}\label{subsec-Alc}
For any positive root $\alpha\in R^+$ and $n\in\DZ$ we define 
$$
 H_{\alpha,n}:=\{\mu\in V\mid \langle \mu,\alpha^\vee\rangle = n\},
 $$
 the {\em affine reflection hyperplane} associated with $\alpha$ and $n$,  and
\begin{align*}
H_{\alpha,n}^+&:=\{\mu\in V\mid \langle \mu, \alpha^\vee\rangle>n\},\\
H_{\alpha,n}^-&:=\{\mu\in V\mid \langle \mu,\alpha^\vee\rangle<n\},
\end{align*}
the corresponding positive and the negative half-space, resp. The {\em affine reflection} with fixed point hyperplane $H_{\alpha,n}$ is 
\begin{align*}
s_{\alpha,n}\colon V&\to V\\
\lambda&\mapsto \lambda-(\langle \lambda,\alpha^\vee\rangle-n)\alpha.
\end{align*}
It maps $H_{\alpha,n}^+$ into $H_{\alpha,n}^-$ and vice versa. 
 \begin{definition} \begin{enumerate}
 \item The connected components of $V\setminus\bigcup_{\alpha\in R^{+},n\in\DZ}H_{\alpha,n}$ are called {\em alcoves}. We denote by $\CA$ the set of alcoves. 
 \item The {\em generic Bruhat order} is the partial order $\preceq$ on $\CA$  that is generated by the relations $A\preceq s_{\alpha,n}(A)$ for $A\in\CA$, $\alpha\in R^+$, $n\in\DZ$ with $A\subset H_{\alpha,n}^-$. 
\end{enumerate}
\end{definition}
The partially ordered set $(\CA,\preceq)$ will provide the topological space underlying the sheaf theory studied in Section \ref{sec-sheavesA}.

\subsection{The affine Weyl group}\label{subsec-affWeyl}
 The {\em affine Weyl group} is the group $\hCW$ of affine transformations on $V$ that is  generated by the affine reflections $s_{\alpha,n}$ for $\alpha\in R^+$ and $n\in\DZ$.  It stabilizes the set of affine hyperplanes and hence permutes the set of alcoves $\CA$. $\CA$ is a principal homogeneous set for this action. 
 
 Denote by $A_e$ the unique alcove that is contained in the dominant Weyl chamber $C^+$ and contains $0$ in its closure. It is called the {\em fundamental alcove}.  
 Then the map $\hCW\to\CA$, $w\mapsto A_w:=w(A_e)$ is a bijection. Via this  bijection, the natural right action of $\hCW$ on itself yields a right action of $\hCW$ on $\CA$ (i.e.  $A_xw:=A_{xw}$).
 Denote by $\hCS\subset\hCW$ the set of reflections along hyperplanes that have a codimension 1 intersection  with the  closure of the fundamental alcove $A_e$. Then $(\hCW,\hCS)$ is a Coxeter system.   
  
    We denote by $t_\lambda\colon V\to V$ the affine translation $\mu\mapsto \lambda+\mu$. Then
\begin{align*}
t_\lambda(H_{\alpha,n})=H_{\alpha,n+\langle\lambda,\alpha^\vee\rangle},\quad
t_\lambda(H_{\alpha,n}^\pm)=H_{\alpha,n+\langle\lambda,\alpha^\vee\rangle}^\pm.
\end{align*}
So  $t_\lambda$ induces  a bijection  $\CA\to\CA$ and it preserves the order $\preceq$, i.e. $
A\preceq B$  if and only if $t_\lambda(A)\preceq t_\lambda(B)$ for all $\lambda\in X$.

Easy calculations yield:
\begin{lemma} \label{lemma-easypeasy}
\begin{enumerate}
\item For $\alpha\in R^+$ and $m,n\in\DZ$ we have $s_{\alpha,n}\circ s_{\alpha,m}=t_{(n-m)\alpha}$.
\item For $\alpha\in R^+$,  $n\in\DZ$ and $\lambda\in X$ we have $s_{\alpha,n}\circ t_\lambda= t_{s_{\alpha,0}(\lambda)}\circ s_{\alpha,n}$.
\end{enumerate}
\end{lemma}

For $\lambda\in\DZ R$, we have $t_\lambda\in\hCW$, as $t_\alpha=s_{\alpha,1}\circ s_{\alpha,0}$ for any $\alpha\in R^+$.


 \subsection{Boxes and special alcoves}
Let $\Delta\subset R^+$ be the set of simple roots.   For $\lambda\in X$ let
$\Pi_\lambda\subset \CA$ be the set of all alcoves $A$ contained in $H_{\alpha,\langle\lambda,\alpha^\vee\rangle-1}^+\cap H_{\alpha,\langle\lambda,\alpha^\vee\rangle}^-$ for any $\alpha\in\Delta$. The $\Pi_\lambda$  are called boxes. Each alcove is contained in a unique box, and each box $\Pi_\lambda$ contains a unique  $\preceq$-maximal alcove $A_\lambda^-$. It is the unique alcove in $\Pi_\lambda$ that contains $\lambda$ in its closure. 
 Following  \cite[\S1.1]{LusAdv} we call an alcove $A$ {\em special} if $A=A_\lambda^-$ for some $\lambda\in X$. We denote by $A_\lambda^+$ the unique alcove that is symmetric to $A_\lambda^-$ with respect to $\lambda$, i.e. it is the unique alcove contained in $C^++\lambda$ that contains $\lambda$ in its closure.

\begin{lemma}\label{lemma-specalc} Let $A\in\CA$ and $\lambda\in X$ with $A\in\Pi_\lambda$. Then we can find  $s_1$, \dots, $s_n\in\hCS$ such that $A=A^-_\lambda s_1\cdots s_n$ and $A\prec A^-_\lambda s_1\cdots s_{n-1}\prec  \cdots \prec A^-_\lambda s_1\prec A^-_\lambda$.
\end{lemma}
\begin{proof} We can find a vector $\delta\in C^-$ and interior points $x\in A^-_\lambda$, $y\in A$ such that $y=x+\delta$ and such that the straight line segment from $x$ to $x+\delta$ does not pass through the intersection of at least two hyperplanes. Hence there exists a sequence of alcoves $A=A_0\prec A_1\prec\ldots A_n=A_\lambda^-$ such that $A_{i}$ and $A_{i+1}$  share a wall for all $i=0,\dots, n-1$. \end{proof}
\subsection{Finite Weyl groups and special sections}\label{subsec-specsec}
Again we fix $\lambda\in X$.
Denote by $\CW_\lambda$ the subgroup of $\hCW$ generated by the reflections $s_{\alpha,\langle\lambda,\alpha^\vee\rangle}$ for $\alpha\in R^+$. These are precisely  the reflections that stabilize $\lambda$. Conjugation with $t_{\lambda-\mu}$ yields an isomorphism $\CW_\lambda\cong\CW_\mu$ for all $\mu\in X$. Denote by $\CS_\lambda\subset\CW_\lambda$ the set of reflections that stabilize (pointwise) a wall of $A_\lambda^-$. Then $(\CW_\lambda,\CS_\lambda)$ is a Coxeter system. We have $\hCW=\CW_\lambda\ltimes\DZ R$.

Set $\CK_\lambda:=\CW_\lambda(A_\lambda^-)\subset\CA$. This is the set of alcoves that contain $\lambda$ in their closure. We call these subsets {\em special sections}. 
\begin{lemma}\label{lemma-specset1}
 \begin{enumerate}
 \item 
The map $\tau=\tau_\lambda\colon \CW_\lambda\to \CK_\lambda$, $\tau(x)=x(A_\lambda^-)$ is a bijection that  identifies the Bruhat order $\le_\lambda$ on $\CW_\lambda$ with the restriction of $\preceq$ to $\CK_\lambda$.  
\item The set $\CK_\lambda$  is a full set of representatives for the $\DZ R$-orbits in $\CA$. 
\end{enumerate}
\end{lemma}
\begin{proof} We prove statement (1). Clearly $\tau$ is a bijection. Since the translations $t_\lambda$ preserve the order $\preceq$ and yield isomorphisms of the finite Coxeter systems,  we can assume $\lambda=0$. The Bruhat order $\le$ on $\CW_0$ is generated by $x< s_{\alpha}x$ if $x^{-1}(\alpha)\in R^+$. This is the case if and only if $\langle x(v),\alpha^\vee\rangle=\langle v,x^{-1}(\alpha^\vee)\rangle<0$ for all $v\in A_0^-$, i.e. $x(A_0^-)\in H_{\alpha,0}^-$, which is the case if and only if $s_{\alpha}x(A_0^-)\in H_{\alpha,0}^-$. Now the statements follows from the definition of the generic Bruhat order.

Statement (2) follows from the facts that $\hCW=\CW_\lambda\ltimes\DZ R$ and that $\CA$ is a principal homogeneous $\hCW$-set. 
\end{proof}

\subsection{Length functions on $\CA$}\label{subsec-lengthfunc}

For any  $\lambda\in X$ let $\ell_\lambda:\CA\rightarrow \DZ$ be the length function relative to $\lambda$ (cf.~\cite{LusAdv}). For an alcove $A$, $\ell_\lambda(A)$ is obtained as follows. Fix a path $\gamma$ in the real vector space $V$ from the alcove $A_\lambda^+$ to $A$ (nothing will depend on the choice of this path). Then $\ell_\lambda(A)=-(m_+ - m_-)$, where $m_+$ is the number of hyperplanes that $\gamma$ crosses from the negative to the positive side, and $m_-$ the number of hyperplanes that $\gamma$ crosses from positive to the negative side.  Note that here we use a convention that is opposite to the one used in \cite{LusAdv} or \cite{SoeRep}.

Clearly,   if $s\in\hCS$ is an affine simple reflection such that $A\prec As$, then $\ell_\lambda(As)=\ell_\lambda(A)-1$.

\subsection{The right action of $\hCS$}
For us, the right action of the elements of $\hCS$ on $\CA$ is of particular importance. Fix $s\in\hCS$. Note that for $A\in\CA$, the elements $A$ and $As$ are always $\preceq$-comparable. An element $A\in\CA$ is called {\em $s$-dominant} if $As\preceq A$, and it is called {\em $s$-antidominant} if it is not $s$-dominant.
\begin{lemma}\label{lemma-propgenord1} Let $A\in\CA$ be $s$-dominant.
\begin{enumerate}
\item The set $\{As,A\}$ is an interval, i.e. for any $B\in\CA$ with $As\preceq B\preceq A$ we have $B=As$ or $B=A$.
\item For $B\in\CA$  we have:
\begin{enumerate} 
\item If $B\preceq A$, then $Bs\preceq A$.
\item If $As\preceq B$, then $As\preceq Bs$.
\end{enumerate}
\end{enumerate}
\end{lemma}
\begin{proof}(1) follows immediately from \cite[(1.4.2)]{LusAdv} and the definition of generic order therein. (2a) is \cite[Proposition 3.2]{LusAdv}, and (2b) is  \cite[Corollary 3.3]{LusAdv}.
\end{proof}

A subset $\CT$ of $\CA$ is called {\em $s$-invariant} if $\CT s\subset\CT$. 
For an arbitrary  subset $\CT$ of $\CA$  define
$$
\CT^\flat:=\CT\cap\CT s, \quad \CT^\sharp:=\CT\cup\CT s.
$$
Then $\CT^\flat$ is the largest $s$-invariant subset of $\CT$, and $\CT^\sharp$ is the smallest $s$-invariant subset of $\CA$ containing $\CT$. 

\begin{definition} We say that a subset $\CJ$ of $\CA$ is {\em open} if  $A\in\CJ$, $B\in\CA$ and $B\preceq A$ imply $B\in\CJ$. 
\end{definition}
This notion is studied in more detail in Section \ref{sec-sheavesparord} in the context of general partially ordered sets. 

\begin{lemma}\label{lemma-propgenord2} Suppose that $\CJ$ is open  in $\CA$.
Then $\CJ^\flat$ and $\CJ^\sharp$ are also open in  $\CA$. 

\end{lemma}
\begin{proof}
 Suppose that  $A\in\CJ^\sharp$, $B\preceq A$ and assume that $B\prec Bs$. We show that $B,Bs\in\CJ^\sharp$.  If $A\in\CJ$, then $B\in\CJ$ as $\CJ$ is open, and hence $Bs\in\CJ s$. If $A\not\in\CJ$, then $A\in\CJ s$, ao $As\in\CJ$  and $As\prec A$. Lemma \ref{lemma-propgenord1} implies $B\preceq As$. But then $B\in\CJ$  and hence $Bs\in\CJ s$. So $\CJ^\sharp$ is open. 

Suppose that $A\in\CJ^{\flat}$ and $B\preceq A$. We want to show that $B\in\CJ^{\flat}$. We have $B\in\CJ$ as $\CJ$ is open, so  we need to show that $Bs\in\CJ$. Interchanging the roles of $A$ and $As$ if necessary we can assume that $As\prec A$. Then Lemma \ref{lemma-propgenord1}   implies that $Bs\preceq A$, hence $Bs\in\CJ$. So $\CJ^\flat$ is open.\end{proof}

Consider the  left action of the group $\DZ R$ on $\CA$. It preserves the partial order and commutes with the right action of $s$, so  it preserves the sets of $s$-dominant and $s$-antidominant elements. Denote by $\CA^{+}$ the subset of $s$-dominant elements in $\CA$, and by $\CA^{-}$ the complement of $s$-antidominant elements.  

Note that  $\CJ^\sharp$, $\CJ^\flat$ and $\CA^{\pm}$ depend on the choice of the simple reflection $s$. We do not incorporate this in the notation, as these symbols are always used   in the context of a fixed choice of $s$.

\begin{lemma} Suppose that $\CJ$ is open. Then  $\CJ^\sharp\setminus\CJ\subset\CA^{+}$ and $\CJ\setminus\CJ^\flat\subset\CA^{-}$. 
\end{lemma}
\begin{proof} Let $A\in\CJ^\sharp\setminus\CJ$. Then $A\in\CJ s$, so $As\in\CJ$ and $A\preceq As$ would imply $A\in\CJ$, a contradiction. Hence $A\in\CA^{+}$. Let $A\in\CJ\setminus \CJ^{\flat}$. Then $As\not\in\CJ$, hence $As\not\preceq A$, so $A\preceq As$ and hence $A\in\CA^{-}$. 
\end{proof}

\subsection{$\alpha$-strings}
Let $\alpha\in R^{+}$ and $A\in\CA$. Then there is a unique $n\in\DZ$ such that 
$
A\subset H_{\alpha,n}^-\cap H_{\alpha,n-1}^+.
$
Define 
$$
\alpha\uparrow A:=s_{\alpha,n}(A).
$$
Then $A\prec\alpha\uparrow A$, and $\alpha\uparrow\cdot\colon \CA\to\CA$ is a bijection. Denote by $\alpha\uparrow^n$ the $n$-fold composition for $n\in\DZ$.


\begin{definition} An {\em $\alpha$-string} in $\CA$ is a subset of $\CA$ of the form $\{\alpha\uparrow^n A\mid n\in\DZ\}$.
\end{definition}
\begin{remarks}
\begin{enumerate}
\item Each $\alpha$-string is a totally ordered subset of $\CA$.
\item The generic Bruhat order $\preceq$ is also generated by the relations $A\preceq \alpha\uparrow A$ for all $A\in\CA$ and $\alpha\in R^+$.
\end{enumerate}
\end{remarks}

Note that for any $\alpha$-string $\Lambda$, the set $\Lambda s$ is an $\alpha$-string as well. 

\begin{lemma} \label{lemma-alphastr} Let $\Lambda$ be an $\alpha$-string and suppose that $\Lambda\cup\Lambda s$ is not a totally ordered subset of $\CA$. Then either $\Lambda$ contains only $s$-antidominant or only $s$-dominant elements.
\end{lemma}

\begin{proof}  Let $A\in\CA$. Suppose that $A$ is $s$-antidominant. It suffices to show that each element in $\Lambda$ is $s$-antidominant.  Note that $\Lambda=\{t_{n\alpha} A\mid n\in\DZ\}\cup\{t_{n\alpha}(\alpha\uparrow A)\mid n\in\DZ\}$. As the translations preserve the order and commute with the right action of $s$ on $\Lambda^\sharp$, it suffices to show that $(\alpha\uparrow A)\preceq(\alpha\uparrow A)s$. But otherwise $(\alpha\uparrow A)s\preceq(\alpha\uparrow A)$, since both  elements are comparable in any case. This leads to a contradiction as follows: as $A\preceq As$, Lemma \ref{lemma-propgenord1} implies $A\preceq(\alpha\uparrow A)s$. If $A=(\alpha\uparrow A)s$, then $\Lambda=\Lambda s$ which contradicts our assumption that $\Lambda^\sharp$ is not totally ordered. Hence $A\prec (\alpha\uparrow A)s$. As $As$ and $(\alpha\uparrow A)s$ are contained in the same $\alpha$-string, they are comparable. Again by Lemma \ref{lemma-propgenord1}, $\{A,As\}$ is an interval, so $As\preceq(\alpha\uparrow A)s$, hence 
$$
A\preceq As\preceq(\alpha\uparrow A)s\preceq(\alpha\uparrow A).
$$
Since the root translations preserve the order, this implies that $\Lambda^\sharp$ is a totally ordered set, contrary to our assumption. Hence $(\alpha\uparrow A)\preceq (\alpha\uparrow A)s$.
\end{proof}


\section{The structure algebra}\label{sec-strucalg}

Let $k$ be a field. We denote by $X^\vee_k=X^\vee\otimes_\DZ k$ the $k$-vector space associated with the lattice $X^\vee$. For convenience, for $v\in X^\vee$, we often denote by $v$ its canonical image  $v\otimes 1$ in $X^\vee_k$.

\begin{definition} We say that $k$ satisfies the {\em GKM-condition} (with respect to $R$), if $\ch\, k\ne 2$ and if any two distinct positive roots $\alpha$ and $\beta$ are $k$-linearly independent, i.e. if  $\alpha\not\in k\beta$ in $X^\vee_k$ for $\alpha\ne\beta$.
\end{definition}

From now on we fix a field $k$ satisfying the above condition. Let $S=S(X^\vee_k)$ be the symmetric algebra of the $k$-vector space $X^\vee_k$. We consider $S$ as a $\DZ$-graded algebra with $X_k^\vee\subset S$ being the homogeneous component of degree 2. For a graded $S$-module $M=\bigoplus_{n\in\DZ}M_n$  the grading shift for $l\in\DZ$ is defined by $M[l]=\bigoplus_{n\in\DZ}M[l]_n$ with $M[l]_n=M_{n+l}$. Homomorphisms between graded modules are assumed to respect the grading. 

Let $\CX$ be a subset of $\CA$.
\begin{definition}\label{def-strucalg} Denote by $\CZ(\CX)$ the set of all $\CX$-tuples $(z_A)\in\prod_{A\in\CX} S$ that satisfy the following properties:
\begin{enumerate}
\item $z_A\equiv z_{s_{\alpha,n}(A)}\mod\alpha^\vee$ for all $A\in\CX$, $\alpha\in R^+$ and $n\in\DZ$ with $s_{\alpha,n}(A)\in\CX$,
\item $z_A=z_{t_\lambda(A)}$ for all $A\in\CX$ and $\lambda\in\DZ R$ with $t_\lambda(A)\in\CX$.
\end{enumerate}
\end{definition}
Then $\CZ(\CX)$  is a commutative, graded, unital $S$-algebra under coordinatewise addition and multiplication (note that the product is taken in the category of graded $S$-modules).
For  $\CX^\prime\subset\CX$ the projection $\prod_{A\in\CX}S\to\prod_{A\in\CX^\prime} S$ along the decomposition restricts to a homomorphism 
$$
\phi_\CX^{\CX^{\prime}}\colon\CZ(\CX)\to\CZ(\CX^\prime).
$$

\subsection{Special sections}
Let $\lambda\in X$. Recall the definition of the special section $\CK_\lambda$ from Section \ref{subsec-specsec}. Note that 
$$
\CZ(\CK_\lambda)=\left\{(z_A)\in\bigoplus_{A\in\CK_\lambda}S\left|\begin{matrix} z_A\equiv z_{s_{\alpha,\langle\lambda,\alpha^\vee\rangle}(A)}\mod\alpha^\vee \\
\text{ for all $A\in\CK_\lambda$, $\alpha\in R^+$}
\end{matrix}
\right\}\right..
$$
\begin{lemma}\label{lemma-isofin} The homomorphism $\CZ\to\CZ(\CK_\lambda)$ is an isomorphism.
\end{lemma}
\begin{proof} As $\CK_\lambda$ is a section for the $\DZ R$-action on $\CA$, the homomorphism is injective. For  $(z_A)\in\CZ(\CK_\lambda)$ and $B\in\CA$ define $z^\prime_B=z_A$ if $B\in A+\DZ R$. We claim that $(z^\prime)$ is contained in $\CZ$. Clearly, property (2) of Definition \ref{def-strucalg} is satisfied.  Let $\alpha\in R^+$ and $n\in\DZ$ and let $A,A^\prime\in\CK_\lambda$ be the representatives of the $\DZ R$-orbits of $B$ and $s_{\alpha,n}(B)$. Then $A^\prime=s_{\alpha,\langle\lambda,\alpha^\vee\rangle}(A)$ by Lemma \ref{lemma-easypeasy}. Hence $z_A\equiv z_{A^\prime}\mod\alpha^\vee$, hence $z_B\equiv z_{s_{\alpha,n}(B)}\mod\alpha^\vee$. It follows that $z^\prime=(z^\prime_B)$ is contained in $\CZ$, and it is, by construction, a preimage of $z$.
\end{proof}
The following innocent looking statement is in fact one of the cornerstones of the theory.
\begin{proposition}\label{prop-Zflab}  Suppose that $\CJ\subset\CK_\lambda$ is an open subset. Then the homomorphism $\CZ\to\CZ(\CJ)$ is surjective.
\end{proposition}
\begin{proof} In order to prove this proposition, we use the theory of Braden--MacPherson sheaves on moment graphs associated to root systems. First, note that the homomorphism $\CZ\to\CZ(\CJ)$ factors over the restriction $\CZ\to\CZ(\CK_\lambda)$. By Lemma \ref{lemma-isofin}, this is an isomorphism. So it is enough to show that the restriction $\CZ(\CK_\lambda)\to\CZ(\CJ)$ is surjective. Let us identify the sets $\CW_\lambda$ and $\CK_\lambda$ via the bijection $\tau_\lambda$ studied in   Lemma \ref{lemma-specset1}. 
  As $(\CW_\lambda,\CS_\lambda)$ is isomorphic to $(\CW_0,\CS_0)$, we can identify $\CZ(\CK_\lambda)$ with 
$$
\CZ(R):=\left\{(z_x)\in\bigoplus_{x\in\CW_0}S\left| \begin{matrix} z_x\equiv z_{s_{\alpha,0}x}\mod\alpha^\vee\\ 
\text{ for all $x\in\CW_0$, $\alpha\in R^+$}
\end{matrix}\right\}\right..
$$
This is the structure algebra  associated with the finite Coxeter system $(\CW_0,\CS_0)$ and its (representation faithful) representation on $X_k$. As $\tau_\lambda$ is order preserving, the set $\CJ$ corresponds to an open subset of $\CW_0$ with respect to the Bruhat order. Then the statement is equivalent to saying that $\CZ(R)$ is flabby in the sense of Definition 3.6 in \cite{FieDuke}. This is the case, by Proposition 4.2. in loc. cit., if and only if $\CZ(R)$ coincides with the global sections of the Braden-MacPherson sheaf $\SB$ on the finite moment graph $\CG$ associated to the above reflection faithful representation. Now the latter statement is true if and only if all the ranks of $\SB$ on vertices of $\CG$ are 1. This fact has been shown in \cite[Corollary 6.1]{LanJoA}.

\end{proof}

In the following  sections we collect some simple facts about the category of $\CZ$-modules.
\subsection{Localizations}\label{sec-locSmod}
Set 
$$
S^\emptyset:=S[\beta^{-1}\mid \text{$\beta$ is a label of $\CG$}]
$$
and for a label $\alpha$ of $\CG$ set
$$
S^{\alpha}:=S[\beta^{-1}\mid \text{$\beta$ is a label of $\CG$ and $\beta\not\in k\alpha$}].
$$
For an $S$-module $M$ we set $M^{\emptyset}:=M\otimes_SS^{\emptyset}$ and $M^{\alpha}:=M\otimes_SS^{\alpha}$. If $M$ is torsion free as an $S$-module, then we have canonical inclusions $M\subset M^{\alpha}\subset M^{\emptyset}$. 

For $\alpha\in R^+$ set
  $$
  \hCW^{\alpha}=\langle s_{\alpha,n}, n\in\DZ\rangle+\DZ R.
  $$
 We denote by $\CA^{\emptyset}:=\CA/\DZ R$ and $\CA^{\alpha}:=\CA/\hCW^{\alpha}$ the corresponding orbit sets.

\begin{lemma} \label{lemma-locZ} 
\begin{enumerate}
\item The inclusion $\CZ\subset\bigoplus_{x\in\CA^{\emptyset}} \CZ(x)$ becomes an isomorphism after applying the functor $\cdot\otimes_S S^{\emptyset}$. 
\item Let $\alpha\in R^+$. Then the inclusion $\CZ\subset\bigoplus_{\Gamma\in\CA^{\alpha}}\CZ(\Gamma)$ becomes an isomorphism after applying the functor $\cdot\otimes_S S^{\emptyset}$. 
\end{enumerate}
\end{lemma}
\begin{proof} As $\CZ$ is torsion free as an $S$-module, the localizations of the inclusions are injective again. Hence we only have to prove that they are surjective. Let us prove part (2). Let $\gamma\in S$ be the product of all $\beta\in R^+\setminus\{\alpha\}$. Then $\gamma$ is invertible in $S^{\alpha}$. Then for any  $z\in\bigoplus_{\Gamma}\CZ(\Gamma)$ the element $\gamma z$ is contained in $\CZ$, hence $z=\gamma^{-1}(\gamma z)$ is contained in $\CZ\otimes_S S^{\alpha}$. Hence the surjectivity. Part (1) is proven in an analogous way. 
\end{proof}

\subsection{Localizations of $\CZ$-modules}\label{subsec-candec}

Here is the definition of a convenient category of $\CZ$-modules.
\begin{definition} We denote by $\CZ\catmod^f$ the full subcategory of the category of $\CZ$-modules that contains all objects $M$ that satisfy the following:
\begin{enumerate}
\item as an $S$-module $M$ is torsion free,
\item  $M=\bigcap_{\alpha\in R^+} M^{\alpha}\subset M^{\emptyset}$.
\end{enumerate}
\end{definition}

Let $M\in\CZ\catmod^f$. Then $M^{\alpha}$ is a $\CZ^{\alpha}$-module, and Lemma \ref{lemma-locZ} yields  a {\em canonical} decomposition
$$
M^{\alpha}=\bigoplus_{\Gamma\in\CA^{\alpha}}M^{\alpha,\Gamma},
$$
where each $M^{\alpha,\Gamma}$ is a $\CZ(\Gamma)$-module.  Analogously there is a canonical decomposition 
$$
M^{\emptyset}=\bigoplus_{x\in\CA^{\emptyset}}M^{\emptyset,x}.
$$
 As $M$ is torsion free as an $S$-module,  we view $M$ as a submodule in $M^{\emptyset}$. The above decomposition  allows us to write each $m\in M$ as an $\CA^{\emptyset}$-tuple $(m_x)$ with $m_x\in M^{\emptyset,x}$. 

\begin{definition} The {\em $\CZ$-support} of   $M$ is $
\supp_\CZ\, M:=\{x\in\CA^{\emptyset}\mid M^{\emptyset,x}\ne 0\}.
$
\end{definition}
\begin{remark}\label{rem-loc} Let $M$ be an object in $\CZ\catmod^f$. Then the localizations  $M^{\alpha}$ for the root $\alpha$ and $M^{\emptyset}$ are objects in $\CZ\catmod^f$ again. Moreover, $M$ can carry at most one $\CZ^{\alpha}$- or $\CZ^{\emptyset}$-structure.
\end{remark}
\subsection{Submodules and quotients}
Let $M\in\CZ\catmod^f$. For a subset  $ T $ of  $\CA^{\emptyset}$ define
\begin{align*}
M_{ T }&:=M\cap\bigoplus_{x\in T } M^{\emptyset,x},\\
M^{ T }&:=M/M_{\CA^{\emptyset}\setminus T }=\im\left(M\subset M^{\emptyset}=\bigoplus_{x\in\CA^{\emptyset}}M^{\emptyset,x}\xrightarrow{p} \bigoplus_{x\in T }M^{\emptyset,x}\right)
\end{align*}
where $p$ denotes the projection along the decomposition. It is easy to see  that $M_{ T }$ and $M^{ T }$ are contained in  $\CZ\catmod^f$ again. Note  that $M_{ T }$ is the largest submodule of $M$ that is supported inside $ T $, and $M^{ T }$ is the largest quotient of $M$ supported inside $ T $. Both constructions are clearly functorial. 
Write  $M^{x,y,\dots,z}$ instead of $M^{\{x,y,\dots,z\}}$. The objects $M^x$ are called the {\em stalks} of $M$.

 The following is a slight, relative generalization of the above construction (which corresponds to the case $N=0$).

\begin{lemma} \label{lemma-subquotZ} Let $ T $ be a subset of $\CA^{\emptyset}$. 
\begin{enumerate}
\item Let $f\colon M\to N$ be a surjective homomorphism in $\CZ\catmod^f$. Then there is a factorization

\centerline{
\xymatrix{
M\ar@/_1pc/[rr]_f\ar[r]^{f_1}&[f, T ]\ar[r]^{f_2}&N
}
}
\noindent
in $\CZ\catmod^f$  such that the following holds:
\begin{enumerate}
\item $f_1$ is surjective and $\supp_\CZ\ker f_1\subset  T $, 
\item $\supp_\CZ\ker f_2\subset\CA^{\emptyset}\setminus T $. 
\end{enumerate}
\item Suppose that $f\colon M\to N$ and $f^\prime\colon M^{\prime}\to  N^{\prime}$ are surjective homomorphisms in $\CZ\catmod^f$.  Suppose that $g\colon M\to M^\prime$ and $h\colon N\to N^\prime$ are  such that $h\circ f=f^\prime\circ g$. Then there exists a unique homomorphism $j\colon [f, T ]\to [f^\prime, T ]$ such that the two squares in the following diagram

\centerline{
\xymatrix{
M\ar@/^1pc/[rr]^f\ar[d]^g\ar[r]_{f_1}&[f, T ]\ar[d]^j\ar[r]_{f_2}&N\ar[d]^h\\
M^\prime\ar@/_1pc/[rr]_{f^\prime}\ar[r]^{f^\prime_1}&[f^\prime, T ]\ar[r]^{f^\prime_2}&N^\prime
}
}
\noindent
commute.
\end{enumerate}
\end{lemma}
\begin{proof} Let  $W\subset M$ be the kernel of $f$, and set $[f, T ]:=M/ W_ T $ and let $f_1\colon M\to M/W_ T $ and $f_2\colon M/W_ T \to N$ be the canonical homomorphisms. We claim $M/W_ T $  is torsion free as an $S$-module. Indeed, let  $m\in M$ be a preimage of a torsion element $m^\prime$ in $M/W_ T $. Then $f(m)=0$ as $N$ is torsion free, so $m\in W$. Since an $S$-multiple of $M$ is contained in $W_ T $ and since $W/W_ T $  (this is $W^{\CA^{\emptyset}\setminus T }$) is torsion free, we have $m\in W_ T $, so $m^\prime=0$. 
By construction, the kernel $W_ T $ of $f_1$ is supported inside $ T $, and the kernel $W^{\CA^{\emptyset}\setminus T }=W/W_ T $ is supported inside $\CA^{\emptyset}\setminus T $. Hence we proved (1). Statement (2) follows from the functoriality of the construction of $W_ T $.
\end{proof}

\section{Sheaves on partially ordered sets}\label{sec-sheavesparord}
This section provides some  basic results on the topology of partially ordered sets and the theory of sheaves on such topological spaces. 


\subsection{A topology on partially ordered sets}\label{subsec-TopPar} Let us fix a partially ordered set $(\CX,\preceq)$. 
For an element $A$ of $\CX$ we will use the short hand notation $\{\preceq A\}=\{B\in\CX\mid B\preceq A\}$. The notations $\{\succeq A\}$, $\{\prec A\}$, etc. have an analogous meaning. 

Here is the definition of the topology on $\CX$.
\begin{definition} A subset $\CJ$ of $\CX$ is called {\em open}, if $A\in \CJ$ and $B\preceq A$ implies $B\in \CJ$, i.e. if 
$
\CJ=\bigcup_{A\in\CJ}\{\preceq A\}.
$
\end{definition}
This clearly defines a topology on the set $\CX$. The following statements are easy to check. 
\begin{remark}
\begin{enumerate}
\item A subset $\CI$ of $\CX$ is closed if and only if $\CI=\bigcup_{A\in\CI}\{\succeq A\}$.
\item Arbitrary unions and intersections of open sets are open. The same holds for closed sets. 
\item For any subset $ \CT $ of $\CX$, the set 
$ \CT ^-:=\bigcup_{A\in \CT }\{\preceq A\}$ is the smallest open subset that contains $ \CT $, and $ \CT ^+:=\bigcup_{A\in \CT }\{\succeq A\}$ is the smallest closed subset that contains $ \CT $. 
\item A subset $\CK$ is locally closed if and only if $\CK=\CK^+\cap\CK^-$. If $\CK$ is locally closed, then $\CK^-\setminus\CK$ is  open. 
\end{enumerate}
\end{remark}

\subsection{(Pre-)sheaves on partially ordered sets}\label{subsec-PosetSheaves}
Now suppose $\bA$ is an abelian category that has arbitrary products. We will study  (pre-)sheaves on the space $\CX$ with values in $\bA$. Recall that a presheaf $\SM$ on $\CX$ with values in $\bA$ is given by objects $\SM^\CJ$ in $\bA$ for any open subset $\CJ$ of $\CX$, together with restriction morphisms $r_{\CJ}^{\CJ^\prime}\colon \SM^\CJ\to\SM^{\CJ^\prime}$ for inclusions $\CJ^\prime\subset\CJ$ that satify $r_{\CJ^\prime}^{\CJ^{\prime\prime}}\circ r_{\CJ}^{\CJ^\prime}=r_{\CJ}^{\CJ^{\prime\prime}}$ for nested inclusions $\CJ^{\prime\prime}\subset\CJ^\prime\subset\CJ$, and $r_{\CJ}^{\CJ}=\id_{\SM^\CJ}$.  We will often use the notation $m|_{\CJ^\prime}=r_{\CJ}^{\CJ^\prime}(m)$. A morphism $f\colon\SM\to\SN$ of presheaves is given by a collection $f^\CJ\colon \SM^\CJ\to\SN^\CJ$ of morphisms in $\bA$ that commute with the restriction morphisms in the obvious way.

\subsection{Sheafification}
Let $\SM$ be a presheaf on $\CX$. For a  (not necessarily open)  subset $\CT$ of $\CX$ define
$$
\SM^{\{\CT\}}:=\left\{(m_A)\in\prod_{A\in \CT}\SM^{\preceq A}\left|\begin{matrix} m_A|_{\preceq C}=m_B|_{\preceq C}\text{ for all $A,B\in \CT $} \\\text{ and $C\in\CX$ with $C\preceq A$, $C\preceq B$}\end{matrix}\right\}\right..
$$
Note that $\SM^{\{\emptyset\}}=0$.
For $\CT^\prime\subset \CT\subset\CX$, the projection $\prod_{A\in \CT}\SM^{\preceq A}\to\prod_{A\in \CT^\prime}\SM^{\preceq A}$ along the product decomposition obviously yields a morphism
$$
s_{\CT}^{\CT^\prime}\colon \SM^{\{\CT\}}\to \SM^{\{\CT^\prime\}}.
$$

 As any open subset $\CJ$ of $\CX$ has the unique finest open covering $\CJ=\bigcup_{A\in\CJ}\{\preceq A\}$, the sheafification of $\SM$ is the sheaf with sections $\SM^{\{\CJ\}}$ for any open subset $\CJ$, and restriction morphisms $s_{\CJ}^{\CJ^\prime}$. The canonical morphism of $\SM$ into its sheafification is the following. 
For any open subset $\CJ$  the direct product of the morphisms $r_{\CJ}^{\preceq A}$ with $A\in\CJ$ yields a morphism
$$
t_{\CJ}\colon\SM^{\CJ}\to \SM^{\{\CJ\}}.
$$
Hence $\SM$ is a sheaf if and only if $t_\CJ$ is an isomorphism for any open subset $\CJ$ of $\CX$. 
 It follows that  a sheaf $\SM$ on $\CX$ is uniquely given by the data of objects $\SM^{\preceq A}$ for all $A\in\bA$ and restriction morphisms $\SM^{\preceq A}\to\SM^{\preceq B}$ for $B\preceq A$ that satisfy the obvious compatibility conditions.

\begin{remark}\label{rem-Tmaxel}  Suppose that  $\CT^\prime\subset \CT$ is such that for all $A\in \CT$ there exists some $B\in \CT^\prime$ with $A\preceq B$. Then $s_{\CT}^{\CT^\prime}\colon \SM^{\{\CT\}}\to \SM^{\{\CT^\prime\}}$ is  an isomorphism (as for $m=(m_A)\in \SM^{\{\CT\}}$ any component $m_A$ is already determined by $m_B$ if $A\preceq B$).
\end{remark}

\begin{lemma}\label{lemma-totord} Let $\CT\subset\CX$ be a finite totally ordered subset. Suppose that $\SM$ is a presheaf on $\CX$ with $\SM^{\emptyset}=\{0\}$ and with the property that for any open set $\CJ$ the restriction morphism $\SM^{\CJ}\to\SM^{(\CJ\cap \CT)^-}$ is an isomorphism. Then $\SM$ is a sheaf supported inside $\CT$.
\end{lemma}
\begin{proof} We show that $\SM$ coincides with its sheafification, i.e. that $t_\CJ\colon \SM^{\CJ}\to \SM^{\{\CJ\}}$ is an isomorphism for any open subset $\CJ$. Fix $\CJ$. As $\CT$ is totally ordered and finite, $(\CJ\cap \CT)^-$ is either empty or of the form $\{\preceq A\}$ for some $A\in\CJ\cap \CT$. In the first case, $\SM^{\CJ}=\SM^{\{\CJ\}}=0$. In the second case, $\SM^{\CJ}\cong \SM^{\preceq A}$ and $\SM^{\{\CJ\}}\cong\SM^{\preceq A}$ (by Remark \ref{rem-Tmaxel}) and $t_\CJ$ is obviously an isomorphism. If $A\not\in\CT$, then $\SM^{\preceq A}\cong\SM^{\prec A}$, hence $A$ is not contained in the support of $\SM$.
\end{proof}
\subsection{Subquotients and the $\preceq$-support}

Recall the following notion of sheaf theory. 
\begin{definition} A presheaf $\SM$ on $\CX$ is called {\em flabby}, if for any inclusion $\CJ^\prime\subset\CJ$ of open subsets of $\CX$ the restriction morphism $\SM^{\CJ}\to\SM^{\CJ^\prime}$ is surjective.
\end{definition}

Let $\SM$ be a  presheaf on $\CX$ with values in $\bA$, and let $\CK$ be a locally closed subset of $\CX$. Recall that we denote by $\CK^-$ the smallest open subset that contains $\CK$, and that $\CK^-\setminus\CK$ is open again. 

\begin{definition} Define  $
\SM_{[\CK]}$ as the kernel of the restriction morphism $r_{\CK^-}^{\CK^-\setminus\CK}\colon\SM^{\CK^-}\to \SM^{\CK^-\setminus\CK}$.
\end{definition}
In sheaf theory language, those are the sections of $\SM$ supported inside $\CK$. Note that a singleton $\{A\}$ is always locally closed (with $\CK^-=\{\preceq A\}$ and $\CK^-\setminus\CK=\{\prec A\}$). We slightly simplify notation and write $\SM_{[A]}$ instead of $\SM_{[\{A\}]}$.
\begin{definition} Define the {\em $\preceq$-support} of  a presheaf $\SM$ as  $\supp_{\preceq}\, \SM:=\{A\in\CX\mid \SM_{[A]}\ne 0\}$.
\end{definition}
We warn the reader that this is {\em not} the sheaf theoretic support of $\SM$, as the latter would be the set of all $A$ with $\SM^{\preceq A}\ne 0$ (as $\SM^{\preceq A}$ is the sheaf theoretic stalk of $\SM$ at $A$). 
\begin{definition}
A presheaf $\SM$ on $\CX$ is said to be {\em finitely supported} if  $\SM_{[A]}=0$ for all but finitely many $A\in\CX$ and if the set $\{C\in\CX\mid\SM^{\preceq C}\ne 0\}$ is locally bounded from below.
\end{definition}
Recall that a subset $\CT $ of $\CX$ is called {\em locally bounded from below}, if there are $A_1,\dots,A_n\in\CX$ such that $ \CT \subset\bigcup_{i=1,\dots,n}\{\succeq A_i\}$. 

Note that while it is convenient for us to state the above definitions in the context of presheaves, they really make more sense in the category of sheaves. 
The following is  easy to prove.
\begin{remarks}\label{rem-subquot}
\begin{enumerate}
\item  Suppose that  $\SM$ is a flabby sheaf on $\CX$. Let $\CJ^\prime\subset\CJ\subset\CX$ be open sets and set $\CK=\CJ\setminus\CJ^\prime$. Then the restriction morphism $\SM^{\CJ}\to\SM^{\CK^-}$ induces an isomorphism 
$$
\ker(\SM^{\CJ}\to \SM^{\CJ^\prime})\cong \SM_{[\CK]}.
$$
\item 
Let $\SM$ be a finitely supported flabby  sheaf on $\CX$ and let $\CK$ be a locally closed subset of $\CX$. Let  $\CK\cap\supp_{\preceq}\SM=\{A_1,\dots,A_n\}$ be an enumeration such that $A_j\preceq  A_i$ implies $j\le i$.  Then there is a cofiltration  $\SM_{[\CK]}=\SM_0\supset \SM_1\supset\dots\supset \SM_{n+1}=0$  in $\bA$ with $\SM_{i-1}/\SM_{i}\cong\SM_{[A_i]}$ .
\end{enumerate}
\end{remarks}

\begin{lemma}\label{lemma-fin} Let $\SM$ be a flabby, finitely supported sheaf. Let $\CJ^\prime\subset\CJ$ be  open subsets of $\CX$ with $\CJ^\prime\cap\supp_{\preceq}\SM=\CJ\cap\supp_{\preceq}\SM$. Then the restriction morphism $\SM^{\CJ}\to\SM^{\CJ^\prime}$ is an isomorphism.
\end{lemma}
\begin{proof} As $\SM$ is flabby, the restriction morphism $\SM^{\CJ}\to\SM^{\CJ^\prime}$ is surjective. By  Remark \ref{rem-subquot},  its kernel is $\SM_{[\CJ\setminus\CJ^\prime]}$ and this  is an extension of the objects $\SM_{[A]}$ with $A\in\CJ\setminus\CJ^\prime$. But all those objects are trivial as $(\CJ\setminus\CJ^\prime)\cap \supp_{\preceq}\SM=\emptyset$.  Hence $\SM_{[\CJ\setminus\CJ^\prime]}=\{0\}$. 
\end{proof}

\subsection{Exact sequences}
The category of sheaves on $\CX$ with values in $\bA$ is abelian, hence we have a notion of short exact sequences.
\begin{lemma} \label{lemma-ses}
Let $0\to\SA\to\SB\to\SC\to 0$ be a sequence of finitely supported flabby sheaves on $\CX$. Then the following statements are equivalent:
\begin{enumerate}
\item The sequence is exact.
\item For any open subset $\CJ$ of $\CX$, the sequence
$$
0\to\SA^\CJ\to\SB^\CJ\to\SC^\CJ\to 0
$$
is exact in $\bA$.
\item For any $A\in\CX$, the sequence 
$$
0\to\SA^{\preceq A}\to\SB^{\preceq A}\to\SC^{\preceq A}\to 0
$$
is exact in $\bA$.
\item For any $A\in\CX$, the sequence
$$
0\to\SA_{[A]}\to\SB_{[A]}\to\SC_{[A]}\to 0
$$
is exact in $\bA$.
\end{enumerate}
\end{lemma}
\begin{proof}  Recall that a sequence of sheaves is exact if and only if the induced sequences on all stalks are exact. Hence (1) and (3) are equivalent. Moreover, (3) is a special case of (2).  And (1) implies (2) by standard arguments (using the flabbiness of $\SA$). Hence (1), (2) and (3) are equivalent. The snake lemma proves that (2) implies (4). For a finitely supported sheaf $\SM$, the object $\SM^{\CJ}$ is a (finite) extension of the subquotients $\SM_{[A]}$ by Remark \ref{rem-subquot}, hence (4) implies (2). 
\end{proof}

\section{Sheaves on $(\CA,\preceq)$} \label{sec-sheavesA} 

Consider the set $\CA$ of alcoves as a topological space with the topology induced by the generic Bruhat order $\preceq$.  In the following we consider sheaves on the topological space $\CA$ with values in the abelian category $\CZ\catmod$ of $\CZ$-modules. The sheaves that we are most interested in, however, satisfy some conditions, the most important of those are the {\em support condition} and the {\em local extension condition}.

\subsection{The support condition}
Note that we defined two notions of support. One is the $\CZ$-support of an object in $\CZ\catmod^f$. This is a subset of the set $\CA^\emptyset$ of $\DZ R$-orbits in $\CA$. The other the $\preceq$-support of a sheaf of $\CZ$-modules on $\CA$, and this is a subset of $\CA$. The following definition relates these two notions.
Let $\SM$ be a sheaf  on $\CA$ with values in $\CZ\catmod^f$. Denote by $\pi\colon\CA\to\CA^{\emptyset}$ the orbit map.

\begin{definition} We say that $\SM$ {\em satisfies the support condition} if for any locally closed subset $\CK$ of $\CA$ the $\CZ$-module  $\SM_{[\CK]}$ is $\CZ$-supported inside $\pi(\CK)$. 
\end{definition}

\begin{remark}\label{rem-suppcond} Remark \ref{rem-subquot}, (2), implies the following. If $\SM$ is a finitely supported flabby sheaf, then it satisfies the support condition if and only if for any $A\in\CA$  the element  $z=(z_B)\in\CZ$ acts on $\SM_{[A]}$ via multiplication by $z_A$.
\end{remark}

\subsection{The localization of sheaves}
Let $\alpha\in R^+$ and let $\SM$ be a presheaf of $\CZ$-modules on $\CA$. Define a presheaf $\SM^\alpha$ of $\CZ^\alpha$-modules on $\CA$ by simply setting $(\SM^\alpha)^\CJ:=\SM^{\CJ}\otimes_S S^\alpha$ and $(r^{\SM^{\alpha}})_{\CJ}^{\CJ^\prime}=(r^\SM)_{\CJ}^{\CJ^\prime}\otimes 1$ for an inclusion $\CJ^\prime\subset\CJ$ of open subsets. We define the presheaf $\SM^{\emptyset}$ of $\CZ^{\emptyset}$-modules analogously. 

\begin{lemma} \label{lemma-locsheaf} Let $\SM$ be a presheaf on $\CA$ with values in $\CZ\catmod^f$.
\begin{enumerate}
\item Suppose that $\SM$ is a finitely supported sheaf. Then $\SM^{\emptyset}$ and $\SM^{\alpha}$ for any $\alpha\in R^+$ are sheaves.
\item If $\SM^{\alpha}$ is a sheaf for each $\alpha\in R^+$, then $\SM$ is a sheaf. 
\end{enumerate} 
\end{lemma}
\begin{proof}  We prove part (1). So let  $\SM$ be a finitely supported sheaf. Let $\alpha\in R^+$. In order to show that $\SM^{\alpha}$ is a sheaf, we show that it coincides with its sheafification, i.e. that for any open subset  $\CJ$  the homomorphism $t_\CJ\colon(\SM^{\alpha})^\CJ\to(\SM^{\alpha})^{\{\CJ\}}$  is an isomorphism. First, observe that for any finite set $\CF$ the canonical homomorphism $(\SM^{\{\CF\}})^{\alpha}\to(\SM^{\alpha})^{\{\CF\}}$ is an isomorphism. Now fix $\CJ$ and let $\CT=\CJ\cap\supp_{\preceq}\SM$. By the definition of support and since $\SM$ is a sheaf, the homomorphism $\SM^{\CJ}\to\SM^{\{\CT\}}$, $m\mapsto (m|_{\preceq A})_{A\in \CT}$, is an isomorphism. Applying the functor $(\cdot)^{\alpha}$ hence yields an isomorphism $(\SM^{\CJ})^\alpha=(\SM^{\alpha})^\CJ\to (\SM^{\{\CT\}})^\alpha=(\SM^{\alpha})^{\{\CT\}}$. So each local section of $\SM^{\alpha}$ is uniquely determined by its restrictions to $\{\preceq A\}$ with $A$ in $\supp_{\preceq}\SM$, and the above isomorphism hence implies that $\SM^{\alpha}$ is a sheaf.
With analogous arguments  we prove that $\SM^{\emptyset}$ is a sheaf.

Now let us prove (2). First note that if $\SM^{\alpha}$ is a sheaf for some $\alpha$, then $\SM^{\emptyset}$ is a sheaf as well (by the arguments above). Suppose we are given local sections $m_i\in\SM^{\CJ_i}$ for $i\in I$ with $m_i|_{\CJ_i\cap\CJ_j}=m_j|_{\CJ_i\cap\CJ_j}$. These can be considered as sections in $(\SM^{\alpha})^{\CJ_i}$. So if $\SM^{\alpha}$ is a sheaf, there is a unique section $m^{\alpha}\in(\SM^{\alpha})^\CJ$ (with $\CJ=\bigcup_{i\in I}\CJ_i$) extending all $m_i$. But the sections $m^{\alpha}$ must coincide as sections in $(\SM^{\emptyset})^\CJ$, as $\SM^{\emptyset}$ is a sheaf. Hence this section is already contained in $\SM^{\CJ}=\bigcap_{\alpha\in R^+}(\SM^{\alpha})^\CJ$.
 \end{proof}

\subsection{The local extension condition}
Let $\SM$ be a sheaf on $\CA$ with values in $\CZ\catmod^f$. The results in Section \ref{subsec-candec} imply that there is a canonical decomposition $\SM^{\alpha}=\bigoplus_{\Gamma\in\CA^{\alpha}} \SM^{\alpha,\Gamma}$, where $\SM^{\alpha,\Gamma}$ is a sheaf of $\CZ(\Gamma)$-modules. For our category, we require a finer decomposition. Note that each $\hCW^{\alpha}$-orbit $\Gamma$ is a disjoint union of $\alpha$-strings. 

\begin{definition} We say that $\SM$ {\em satisfies the local extension condition} if for any $\alpha\in R^+$ the sheaf $\SM^\alpha$ splits into a direct sum $\SM^{\alpha}=\bigoplus_{i=1}^n \SM^{\alpha}_i$ in such a way that each $\SM^{\alpha}_i$ is $\preceq$-supported inside a single $\alpha$-string in $\CA$.
\end{definition}

\subsection{Categories of sheaves} Now we are ready to define the categories of sheaves that will play a decisive role in the following.

\begin{definition} Denote by $\bC$ the full subcategory of the category of sheaves on $\CA$ with values in $\CZ\catmod$ that contains all objects $\SM$ that satisfy the following:
\begin{enumerate}
\item for any open subset $\CJ$ of $\CA$, $\SM^{\CJ}$ is an object in $\CZ\catmod^f$,
\item $\SM$ is  flabby and  finitely supported,
\item $\SM$  satisfies the support condition and the local extension condition. 
\end{enumerate}
\end{definition}
\begin{remark} Let $\SM$ be an object in $\bC$. As we do not assume any finiteness condition on the $\CZ$-modules $\SM^\CJ$ for open $\CJ$, we can view the localized sheaf $\SM^{\alpha}$ for $\alpha\in R^+$ or $\SM^{\emptyset}$ as objects in $\bC$ as well. Note that by Remark \ref{rem-loc} an object $\SM$ can carry at most one structure as a sheaf of $\CZ^{\alpha}$- or $\CZ^{\emptyset}$-modules.
\end{remark}
Sheaves that admit a Verma flag are defined as follows: 
\begin{definition}
We say that an object $\SM$ of $\bC$ {\em admits a Verma flag} if for any open subset $\CJ$ the object $\SM^{\CJ}$ is graded free of finite rank as an $S$-module. We denote by $\bB$ the full subcategory of $\bC$ that contains all objects that admit a Verma flag.
\end{definition}

\begin{lemma}\label{lemma-Vermasubquot}  The object $\SM$ in $\bC$ admits a Verma flag if and only if  $\SM_{[A]}$ is a graded free $S$-module for any $A\in\CA$. 
\end{lemma}
\begin{proof} By Remark \ref{rem-subquot} each local section of   $\SM$ is an extension of the subquotients $\SM_{[A]}$ with $A\in\CJ$. Hence local sections are graded free as $S$-modules if and only if all subquotients are graded free as $S$-modules. \end{proof}

\section{Sheaves associated with the structure algebra}\label{sec-pfspec}
Let $\CK\subset\CA$ be a locally closed subset. Define a sheaf $\ul\CK$ of $\CZ$-modules on $\CA$ in the following way. For any open subset $\CJ$ of $\CA$ set 
$$
\ul{\CK}^{\CJ}:=\CZ(\CK\cap\CJ),
$$
 and use the canonical homomorphism $\CZ(\CK\cap\CJ)\to\CZ(\CK\cap\CJ^\prime)$ for two open subsets $\CJ^\prime\subset\CJ$ as the restriction homomorphism. This defines  a presheaf $\ul{\CK}$ of $\CZ$-modules on $\CA$. From the locality of the definition one immediately deduces that $\ul{\CK}$ is a sheaf. For $\CK^\prime\subset\CK$ the canonical homorphisms $\CZ(\CK\cap\CJ)\to\CZ(\CK^\prime\cap\CJ)$ combine and yield a homomorphism $\ul{\CK}\to\ul{\CK^\prime}$ of sheaves. 
\subsection{Standard objects in $\bC$}
In the special case $\CK=\{C\}$ we set $\SV(C):=\ul{\{C\}}$. For any open subset $\CJ$ we hence have
$$
\SV(C)^{\CJ}:=\begin{cases}
S,&\text{ if $C\in \CJ$},\\
0,&\text{ else}
\end{cases}
$$
as a graded $S$-module, and the $\CZ$-module structure is such that $z=(z_A)$ acts as multiplication with $z_C\in S$. The restriction maps are the obvious (non-trivial) ones. It is easy to check that it satisfies the properties required for objects in $\bC$. Clearly it admits a Verma flag, hence is contained in $\bB$.

\subsection{Sheaves associated with special sections}
Fix $\lambda\in X$ and recall the special section $\CK_\lambda$. It is a locally closed subset of $\CA$.
\begin{proposition}\label{prop-ZKinC}  
\begin{enumerate}
\item The sheaf $\ul{\CK_\lambda}$ is an object in $\bB$.
\item For $A\in\CA$ we have
$$
\ul{\CK_\lambda}_{[A]}\cong
\begin{cases}
0,&\text{ if $A\not\in\CK_\lambda$},\\
S[-2l_A],&\text{ if $A\in\CK_\lambda$},
\end{cases}
$$
where $l_A$ is the number of $\alpha\in R^+$ with $s_{\alpha,\langle\lambda,\alpha^\vee\rangle}(A)\preceq A$.
\item $\ul{\CK_\lambda}$ admits a Verma flag, is indecomposable and admits an epimorphism onto $\SV(A_\lambda^-)$. 
\end{enumerate}
\end{proposition}
\begin{proof}  It follows from Proposition \ref{prop-Zflab} that $\ul{\CK_\lambda}$ is a flabby sheaf. 
Let $A\in\CA$. If $A\not\in\CK_\lambda$, then $\{\preceq A\}\cap\CK_\lambda=\{\prec A\}\cap\CK_\lambda$ and hence $\ul{\CK_\lambda}_{[A]}=0$. If $A\in\CK_\lambda$, then the kernel $\ul{\CK_\lambda}^{\preceq A}\to \ul{\CK_\lambda}^{\prec A}$ is the set of all $r=(r_C)_{C\in\CK\cap\{\preceq A\}}$ with $r_C=0$ unless $C=A$. Hence $\ul{\CK_\lambda}_{[A]}=(\prod \alpha)S$, where the product is taken over all $\alpha\in R^+$ such that there is some $n\in\DZ$ with $s_{\alpha,n}(C)\in\CK_\lambda\cap\{\preceq A\}$. Now $s_{\alpha,n}(C)\in\CK_\lambda\cap\{\preceq A\}$ is equivalent to $n=\langle\lambda,\alpha^\vee\rangle$ and $s_{\alpha,n}(A)\preceq A$. Hence  $\ul{\CK_\lambda}_{[A]}\cong S[-2l_A]$, so (2). It follows furthermore that $\ul{\CK_\lambda}$ is finitely supported. The definition of the $\CZ$-module structure implies directly that $\ul{\CK_\lambda}$ satisfies the support condition. 

Let $\alpha\in R^+$. As the image of $\beta^\vee$ is invertible in $S^{\alpha}$ for all  $\beta\in R^+$, $\beta\ne\alpha$, we have
$$
(\ul{\CK_\lambda}^\CJ)^\alpha=\left\{(r_A)\in\bigoplus_{A\in\CK\cap\CJ} S^\alpha\left| \begin{matrix} r_A\equiv r_{s_{\alpha,n}(A)}\mod\alpha^\vee\text{ for all $A\in\CK\cap\CJ$} \\ \text{ $n\in\DZ$ with $s_{\alpha,n}(A)\in\CK\cap\CJ$}\end{matrix}\right\}\right..
$$
As the relations are hence only between components indexed by elements in the same $\alpha$-string, it follows that $\ul{\CK_\lambda}$ satisfies the local extension condition as well. Hence (1). 

Finally, it is immediate from the definition that each local section of $\ul{\CK_\lambda}$ is an object in $\CZ\catmod^f$. Together with the above it follows that $\ul{\CK_\lambda}$ is contained in $\bC$. Lemma \ref{lemma-Vermasubquot} together with the above implies that $\ul{\CK_\lambda}$ admits a Verma flag.  The global sections of $\ul{\CK_\lambda}$ are a free $\CZ$-module of rank one by Lemma \ref{lemma-isofin}. Since $\CZ$ is  indecomposable,  the sheaf $\ul{\CK_\lambda}$ is indecomposable as well. Finally, as $A_\lambda^-$ is a $\preceq$-minimal element in the support of $\ul{\CK_\lambda}$, there is an epimorphism $\ul{\CK_\lambda}\to\SV(A_\lambda^-)$.
\end{proof}

\section{Projective objects in $\bC$}\label{sec-projobj} 
The category $\bC$ is not abelian (as, for example, its objects are supposed to be torsion free, which prevents general quotients). However, it inherits an exact structure from the surrounding category of sheaves of $\CZ$-modules on $\CA$, and the equivalent statements for short exact sequences in Lemma \ref{lemma-ses} hold. Hence we have a notion of projective objects. 
 
 \subsection{A sufficient condition on projectivity}

\begin{proposition} \label{prop-projobj} Let $\SB$ be an object in $\bC$ with the following properties:
\begin{enumerate}
\item For each $x\in \CA^{\emptyset}$ and each $A\in x$, the stalk $(\SB^{\preceq A})^{x}$ is  projective  in the category of graded $S$-modules.
\item For each $\SM\in\bC$ and all $A\in\CA$ the following holds: Suppose that we are given for each $B\prec A$ a homomorphism $h^{\preceq B}\colon \SB^{\preceq B}\to \SM^{\preceq B}$ such that for all $C\prec B\prec A$ the diagram

\centerline{
\xymatrix{
 \SB^{\preceq B}\ar[d]
 \ar[r]^{h^{\preceq B}}&\SM^{\preceq B}\ar[d]
 \\
  \SB^{\preceq C}\ar[r]^{h^{\preceq C}}&\SM^{\preceq C}
 }
 }
 \noindent 
commutes. Then there is a homomorphism $h^{\preceq A}\colon\SB^{\preceq A}\to\SM^{\preceq A}$ of $\CZ$-modules such that the diagrams

\centerline{
\xymatrix{
 \SB^{\preceq A}\ar[d]
 \ar[r]^{h^{\preceq A}}&\SM^{\preceq A}\ar[d]
 \\
  \SB^{\preceq B}\ar[r]^{h^{\preceq B}}&\SM^{\preceq B}
 }
 }
 \noindent 
  commute for all $B\prec A$. 
 \end{enumerate}
Then $\SB$ is projective in $\bC$.
\end{proposition}
The verticals in the above diagrams are the restriction homomorphisms. 
\begin{proof} Let $f\colon \SM\sur \SN$ be an epimorphism in $\bC$ and let $g\colon \SB\to \SN$ be arbitrary morphism in $\bC$. We want to construct a morphism $h\colon \SB\to\SM$ with $g=f\circ h$. We need to construct homomorphisms $h^{\preceq A}\colon \SB^{\preceq A}\to \SM^{\preceq A}$ of $\CZ$-modules that are compatible with the restriction homomorphisms of $\SB$ and $\SM$. We do this  by induction on $A$. As the set $\{B\in\CA\mid \SB^{\preceq B}\ne 0\}$ is locally bounded from below, we can assume that we have already constructed $h^{\preceq B}$ for all $B\prec A$. By our assumption, there is a homomorphism $\tilde h^{\preceq A}\colon \SB^{\preceq A}\to \SM^{\preceq A}$ that extends all $h^{\preceq B}$. But note that we do not necessarily have $g^{\preceq A}=f^{\preceq A}\circ \tilde h^{\preceq A}$. However, the image of the difference $l:=g^{\preceq A}-f^{\preceq A}\circ \tilde h^{\preceq A}\colon\SB^{\preceq A}\to\SN^{\preceq A}$ is contained  in the kernel $\SN_{[A]}$ of the restriction $\SN^{\preceq A}\to\SN^{\prec A}$. In order to correct our choice, it is hence sufficient to find a homomorphism $h^\prime\colon\SB^{\preceq A}\to\SM_{[A]}$ such that $l=f^{\preceq A}|_{\SM_{[A]}}\circ h^\prime$, since then $h^{\preceq A}:=\tilde h^{\preceq A}+h^\prime$ is such that $g^{\preceq A}=f^{\preceq A}\circ h^{\preceq A}$.

Let $x=\pi(A)$. By the support condition, $\SN_{[A]}$ is supported as a $\CZ$-module inside $\{x\}$, and hence there is a homomorphism $\tilde l\colon (\SB^{\preceq A})^x\to \SN_{[A]}$ such that the diagram

\centerline{
\xymatrix{
\SB^{\preceq A}\ar[dr]_{\pr_x}\ar[rr]^l&&\SN_{[A]}\\
&(\SB^{\preceq A})^x\ar[ur]_{\tilde l}&\\
}
}
\noindent 
commutes. As $f$ is an epimorphism in $\bC$, $f_{[A]}\colon \SM_{[A]}\to \SN_{[A]}$ is surjective (cf. Lemma \ref{lemma-ses}). As $(\SB^{\preceq A})^x$ is projective in the category of graded $S$-modules,  we can find a homomorphism $r\colon(\SB^{\preceq A})^x\to\SM_{[A]}$ such that the diagram

\centerline{
\xymatrix{
&(\SB^{\preceq A})^x\ar[dr]^{\tilde l}\ar[dl]_r&\\
\SM_{[A]}\ar[rr]^{f_{[A]}}&&\SN_{[A]}
}
}
\noindent
commutes.  Now we can take for $h'$ the composition $\SB^{\preceq A}\xrightarrow{\pr_x}(\SB^{\preceq A})^x\xrightarrow{r}\SM_{[A]}$.
\end{proof}
\subsection{Special projectives}
Our next goal is to show that the sheaves $\ul{\CK_\lambda}$ are projective objects in the categories $\bC$ and $\bB$. For this, we will need the following result.

\begin{lemma}\label{lemma-specset2}
For any alcove $A$  and any $\lambda\in X$ with $\{\preceq A\}\cap\CK_\lambda\ne\emptyset$ we have
$\pi(\{\preceq A\}\cap\CK_\lambda)=\pi(\{\prec A\}\cap\CK_\lambda)\cup\{\pi(A)\}$.
\end{lemma}
\begin{proof} We only have to show that $\pi(A)\in\pi(\{\preceq A\}\cap\CK_\lambda)$. 
Note that there is a unique $\gamma\in\DZ R$ with $A\in\CK_{\lambda}+\gamma=\CK_{\lambda+\gamma}$. We claim that $\gamma\ge 0$. This serves our purpose, since then $A-\gamma\in\CK_\lambda$ and $A-\gamma\preceq A$. 

Pick an element $B\in\CK_\lambda$ with $B\preceq A$. Then there exist $\alpha_1, \ldots, \alpha_r\in R^+$ and $n_1, \ldots, n_r\in\DZ$ such that 
\[A\succ A_1\succ  A_2\succ \ldots\succ  A_r=B,\]
where $A_i= s_{\alpha_i, n_i}s_{\alpha_{i-1}, n_{i-1}}\ldots  s_{\alpha_1, n_1}(A)$. Denote by $x:=s_{\alpha_r, n_r}s_{\alpha_{r-1}, n_{r-1}}\ldots  s_{\alpha_1, n_1}.$
By definition of the generic order (by induction on $r$), it is easy to see that if $v$ is in the closure of $A$, then $v- x(v)\in\DR_{\geq 0}R^+$. Since $x(\lambda+\gamma)=\lambda$, we deduce $\gamma\in\DR_{\geq 0} R^+$ and, hence, $\gamma\geq 0$. 
\end{proof}

\begin{proposition}\label{prop-Oproj} The object $\ul{\CK_\lambda}$ is projective in $\bC$. 
\end{proposition}

\begin{proof} We show that $\SB:=\ul{\CK_\lambda}$ satisfies the sufficient conditions listed in Proposition \ref{prop-projobj}. By  Proposition \ref{prop-Zflab} for any open subset $\CJ$ of $\CA$ we have a surjective homomorphism $\CZ\to\SB^\CJ$. As the stalks, by definition, are torsion free as $S$-modules, it follows that  $(\SB^{\preceq A})^x$ either vanishes or is graded free of rank $1$ as an $S$-module. In any case the stalk is projective in the category of graded $S$-modules.

We need to check the second condition. So  let $A\in\CA$  and let $\SM$ be an object in $\bC$. Suppose  that we are given  homomorphisms $h^{\preceq B}\colon\SB^{\preceq B}\to \SM^{\preceq B}$ of $\CZ$-modules for all $B\prec A$ and that these data are compatible with the restriction homomorphisms. We have to find a homomorphism $h^{\preceq A}\colon \SB^{\preceq A}\to \SM^{\preceq A}$ extending these data. We distinguish two cases. If $\{\preceq A\}\cap\CK_\lambda=\emptyset$, then $\SB^{\preceq B}=0$ for all $B\preceq A$, so $h^{\preceq A}=0$ serves our purpose. 

Suppose that  $\{\preceq A\}\cap\CK_\lambda\ne \emptyset$. 
Now the homomorphisms $h^{\preceq B}$ for $B\prec A$  combine and yield a homomorphism $h^{\prec A}\colon\SB^{\prec A}\to\SM^{\prec A}$ of $\CZ$-modules and we need to find $h^{\preceq A}\colon\SB^{\preceq A}\to\SM^{\preceq A}$ such that  the diagram

\centerline{
\xymatrix{
\SB^{\preceq A}\ar[r]^{h^{\preceq A}}\ar[d]_{r_{\preceq A}^{\prec A}}&\SM^{\preceq A}\ar[d]^{r_{\preceq A}^{\prec A}}\\
\SB^{\prec A}\ar[r]^{h^{\prec A}}&\SM^{\prec A}.
}
}
\noindent
commutes.
We claim that $Q:=(r_{\preceq A}^{\prec A})^{-1}(h^{\prec A}(\SB^{\prec A}))$ is $\CZ$-supported inside $\pi(\{\preceq A\}\cap\CK_\lambda)$. Indeed, the image of $h^{\prec A}$ is $\CZ$-supported inside $\pi(\{\prec A\}\cap\CK_\lambda)$, as $\SB^{\prec A}$ is, and the kernel of $r_{\preceq A}^{\prec A}$ is supported inside $\pi(\{A\})$, as $\SM$ satisfies the support condition, and the claim follows from  Lemma \ref{lemma-specset2}.
Now choose a homomorphism $h\colon\CZ\to Q$ of $\CZ$-modules that renders the diagram

\centerline{
\xymatrix{
\CZ\ar[rrr]^h\ar[d]^c&&&Q\ar[d]^{r_{\preceq A}^{\prec A}}\\
\SB^{\prec A}=\CZ^{\pi(\{\prec A\}\cap\CK_\lambda)}\ar[rrr]^{h^{\prec A}}&&&\SM^{\prec A}
}
}
\noindent
commutative ($c$ is the quotient homomorphism). By what we observed above, $h$ must factor over the quotient $\CZ\to\CZ^{\pi(\{\preceq A\}\cap \CK_\lambda)}=\SB^{\preceq A}$, and the resulting homomorphism $h^{\preceq A}\colon \SB^{\preceq A}\to Q\subset\SM^{\preceq A}$ satisfies the conditions mentioned in Proposition \ref{prop-projobj}.

\end{proof}

\section{Wall crossing  functors}\label{sec-wcf}
The main purpose of this section  is to construct a {\em wall crossing functor } $\vartheta_s\colon \bC\to\bC$ for each affine simple reflection $s\in\hCS$. Fix an element $s$ in $\hCS$.

\subsection{The $s$-invariant subalgebra}
Consider the action $\CA\to\CA$, $A\mapsto As$ of $s$ on the right of $\CA$, and define
\begin{align*}
\CZ^s:=\left\{(z_A)\in\CZ\mid z_A=z_{A s}\text{ for all $A\in\CA$}\right\},\\
\CZ^{-s}:=\left\{(z_A)\in\CZ\mid z_A=-z_{A s}\text{ for all $A\in\CA$}\right\}.
\end{align*}
As we can identify $\CZ$ with the structure algebra of the moment graph $\CG$ associated with the root system $R$ in such a way that the right action of $s$  corresponds to the right multiplication of $s$ on $\CW_0$, we have the following  (note that we assume that $2$ is invertible in $k$).

\begin{lemma}[{\cite[Lemma 5.1 \& Proposition 5.3]{FieTAMS}}]\label{lemma-lstrucfree}
\begin{enumerate}
\item We have a decomposition $\CZ=\CZ^s\oplus \CZ^{-s}$, and $\CZ^{-s}$ is a free $\CZ^s$-module of graded rank $v^2$ (i.e. $\CZ^{-s}\cong\CZ^s[-2]$).
\item We have $\supp_{\CZ}(\CZ\otimes_{\CZ^s} N)=(\supp_{\CZ} N)\cup (\supp_{\CZ} N)s$ for any $\CZ$-module $N$ that is torsion free as an $S$-module.  
\end{enumerate}
\end{lemma}
We consider the functor $\epsilon_s:=\CZ\otimes_{\CZ^s}\cdot[1]\colon\CZ\catmod\to\CZ\catmod$. By the above, this is an exact functor (with respect to the natural exact structure on $\CZ\catmod$), and $\epsilon_s M\cong M[1]\oplus M[-1]$ as an $S$-module. It follows that  $\epsilon_s$ restricts to a endofunctor on the category $\CZ\catmod^f$.

\begin{theorem}\label{thm-main} Let $s\in\hCS$. 
\begin{enumerate}
\item There is an up to isomorphism unique functor $\vartheta_s\colon\bC\to\bC$ with the following property. For any $s$-invariant open subset $\CJ$ of $\CA$ we can choose an isomorphism $s_\CJ\colon(\vartheta_s\SM)^\CJ\to \epsilon_s(\SM^\CJ)$ of $\CZ$-modules in such a way that for all inclusions $\CJ^{\prime}\subset\CJ$ of $s$-invariant subsets the diagram

\centerline{
\xymatrix{
(\vartheta_s\SM)^{\CJ}\ar[r]^{s_\CJ}\ar[d]_{r^{\CJ}_{\CJ^\prime}}&\epsilon_s(\SM^\CJ)\ar[d]^{\epsilon_s(r^\CJ_{\CJ^\prime})}\\
(\vartheta_s\SM)^{\CJ^\prime}\ar[r]^{s_{\CJ^\prime}}&\epsilon_s(\SM^{\CJ^\prime})
}
}
\noindent
commutes.
\item The functor $\vartheta_s$ has the following properties:
\begin{enumerate}
\item It is exact.
\item It is self-adjoint, i.e. there is a bifunctorial isomorphism
$$
\Hom(\vartheta_s\SM,\SN)\cong\Hom(\SM,\vartheta_s\SN)
$$ 
for $\SM,\SN\in\bC$.
\item It preserves the category $\bB$ of objects that admit a Verma flag.
\item For $A\in\CA$ with $A\preceq As$ we have functorial isomorphisms
$$
(\vartheta_s\SM)_{[A]}\cong \SM_{[A,As]}[1]\text{ and }(\vartheta_s\SM)_{[As]}\cong \SM_{[A,As]}[-1].
$$
\end{enumerate}
\end{enumerate}
\end{theorem}

We prove the theorem in the following way. First we show that there is at most one functor satisfying the property stated in (1) (up to isomorphism, of course). Then we construct a presheaf $\vartheta_s\SM$ for any object $\SM$ in $\bC$ in a functorial way. 
We show that $\vartheta_s\SM$ is indeed a sheaf, and even an object in $\bC$. From the construction it will be clear that it satisfies the property stated in (1). 
Hence we established the existence of the functor $\vartheta_s$ and finished the proof of (1). Then we show that this functor has all the properties listed in (2).
\subsection{The uniqueness statement}
The uniqueness statement in Theorem \ref{thm-main} follows directly from the slightly more general statement in the next proposition.

\begin{proposition}\label{prop-uniquetrans} Let $\SM $ and $\SN $ be objects in $\bC$ and suppose we are given for any $s$-invariant open subset $\CJ$ of $\CA$ a homomorphism $f^{(\CJ)}\colon \SM ^\CJ\to \SN ^\CJ$ in such a way that for any inclusion $\CJ^\prime\subset\CJ$ of $s$-invariant open subsets of $\CA$ the diagram

\centerline{
\xymatrix{
\SM ^{\CJ}\ar[rr]^{f^{(\CJ)}}\ar[d]_{r^\SM _{\CJ,\CJ^\prime}}&&\SN ^{\CJ}\ar[d]^{r^\SN _{\CJ,\CJ^\prime}}\\
\SM ^{\CJ^\prime}\ar[rr]^{f^{(\CJ^\prime)}}&&\SN ^{\CJ^{\prime}}
}
}
\noindent 
commutes. Then there is a unique homomorphism $f\colon \SM \to \SN $ in $\bC$ such that $f^\CJ=f^{(\CJ)}$ for any $s$-invariant open subset $\CJ$ of $\CA$.
\end{proposition}

\begin{proof}
 Let $\CJ$ be an arbitrary open subset of $\CA$, and denote by $\CJ^{\sharp}$ and $\CJ^{\flat}$ the smallest $s$-invariant subset of $\CA$ containing $\CJ$ and the largest $s$-invariant  subset of $\CJ$. Both sets are open by Lemma \ref{lemma-propgenord2}. By the support condition, the kernel $\SM_{[\CJ^\sharp\setminus\CJ]}$ of the restriction homomorphism $\SM^{\CJ^\sharp}\to\SM^\CJ$ is $\CZ$-supported inside $\pi(\CJ^\sharp\setminus\CJ)\subset\pi(\CA^{+})$, and the kernel of  $\SM_{[\CJ\setminus\CJ^\flat]}$ of $\SM^{\CJ}\to\SM^{\CJ^\flat}$ is $\CZ$-supported inside $\pi(\CJ\setminus\CJ^\flat)\subset\pi(\CA^{-})$. By   Lemma \ref{lemma-propgenord2} these sets are disjoint. Hence the diagram $\SM^{\CJ^\sharp}\to\SM^\CJ\to\SM^{\CJ^\flat}$ identifies with the diagram $\SM^{\CJ^\sharp}\to[r_{\CJ^\sharp}^{\CJ^\flat}, \CA^{+}]\to\SM^{\CJ^\flat}$ constructed in Lemma \ref{lemma-subquotZ}. The analogous statement holds for $\SN$. The functoriality statement in Lemma \ref{lemma-subquotZ} yields the statement of the proposition. 
\end{proof}

\subsection{The construction of $\vartheta_s\SM$}
Let $\SM$ be a sheaf  on $\CA$ with values in $\CZ\catmod^f$.  Let $\CJ$ be an open subset. Let $\CJ^\flat\subset\CJ\subset\CJ^{\sharp}$ be as in the previous section.  The surjective homomorphism $r_{\CJ^\sharp}^{\CJ^\flat}\colon\SM^{\CJ^\sharp}\to\SM^{\CJ^\flat}$ induces a surjective homomorphism $\epsilon_s(r_{\CJ^\sharp}^{\CJ^\flat})\colon\epsilon_s(\SM^{\CJ^\sharp})\to\epsilon_s(\SM^{\CJ^\flat})$ of $\CZ$-modules. Using Lemma \ref{lemma-subquotZ} we construct  $(\vartheta_s\SM)^\CJ$ as the up to isomorphism unique object in $\CZ\catmod^f$ that fits into the diagram

\centerline{
\xymatrix{
\epsilon_s(\SM^{\CJ^\sharp})\ar@/_1pc/[rrrr]_{\epsilon_s(r_{\CJ^\sharp}^{\CJ^\flat})}\ar[rr]^{f_1}&&\vartheta_s\SM^{\CJ}\ar[rr]^{f_2}&&\epsilon_s(\SM^{\CJ^\flat})
}
}
\noindent
with surjective homomorphisms $f_1$ and $f_2$ and 
with $\ker f_1$ $\CZ$-supported inside $\pi(\CA^{+})$ and $\ker f_2$ supported inside $\pi(\CA^{-})$.
For an inclusion $\CJ^\prime\subset\CJ$ of open subsets we have $\CJ^{\prime\flat}\subset\CJ^\flat$ and $\CJ^{\prime\sharp}\subset\CJ^\sharp$. The corresponding restriction homomorphisms $\SM^{\CJ^\sharp}\to\SM^{\CJ^{\prime\sharp}}$ and $\SM^{\CJ^\flat}\to\SM^{\CJ^{\prime\flat}}$ induce homomorphisms  $\epsilon_s(\SM^{\CJ^\sharp})\to\epsilon_s(\SM^{\CJ^{\prime\sharp}})$ and $\epsilon_s(\SM^{\CJ^\flat})\to\epsilon_s(\SM^{\CJ^{\prime\flat}})$. By Lemma \ref{lemma-subquotZ}  there is a unique homomorphism $r_{\CJ}^{\CJ^\prime}$ such that the diagram 

\begin{equation}\label{diag1}
\begin{gathered}
\xymatrix{
\epsilon_s(\SM^{\CJ^{\sharp}})\ar[d]
\ar[r]^{f_1}&(\vartheta_s\SM)^\CJ\ar[d]^{r_{\CJ}^{\CJ^\prime}}\ar[r]^{f_2}&\epsilon_s(\SM^{\CJ^{\flat}})\ar[d]
\\
\epsilon_s(\SM^{\CJ^{\prime\sharp}})\ar[r]^{f_1^\prime}&(\vartheta_s\SM)^{\CJ^\prime}\ar[r]^{f_2^\prime}&\epsilon_s(\SM^{\CJ^{\prime\flat}})
}
\end{gathered}
\end{equation}
\noindent
commutes. 
For a nested inclusion $\CJ^{\prime\prime}\subset\CJ^{\prime}\subset\CJ$ the uniqueness statement in Lemma \ref{lemma-subquotZ} implies 
$r^{\CJ^{\prime\prime}}_{\CJ^{\prime}}\circ r^{\CJ^{\prime}}_{\CJ}=r^{\CJ^{\prime\prime}}_{\CJ}$.
Hence we constructed a presheaf $\vartheta_s\SM$  on $\CA$ with values in $\CZ\catmod^f$. Clearly this construction is functorial, so $\vartheta_s$ is a functor from $\bC$ to the category of presheaves  $\CA$ with values in $\CZ\catmod^f$. We now want to show that $\vartheta_s\SM$ is an object in the category $\bC$ if $\SM$ is. Let's start with some easy to deduce properties.

\begin{lemma}\label{lemma-firstprop} Suppose $\SM$ is a flabby, finitely supported sheaf on $\CA$ with values in $\CZ\catmod^f$. Then the following holds.
\begin{enumerate}
\item  $\vartheta_s\SM$ is a flabby presheaf.
\item Let $\CJ^\prime\subset\CJ$ be  open subsets of $\CA$ with $\CJ^\prime\cap (\supp_{\preceq}\SM)^\sharp=\CJ\cap (\supp_{\preceq}\SM)^\sharp$. Then the restriction homomorphism $(\vartheta_s\SM)^\CJ\to(\vartheta_s\SM)^{\CJ^\prime}$ is an isomorphism.
\end{enumerate}
\end{lemma} 
\begin{proof} 
We prove (1). Let $\CJ^\prime\subset\CJ$ be an inclusion of open subsets of $\CA$. Then the left vertical homomorphism in the diagram (\ref{diag1}) is surjective by the exactness of $\epsilon_s$ and the flabbiness of $\SM$. As $f_1^\prime$ is surjective, so is $r_\CJ^{\CJ^\prime}$.  Hence (1). 

Let us now prove (2). Note that $(\CJ^{\prime})^{\sharp}\subset\CJ^\sharp$ and $(\CJ^{\prime})^{\sharp}\cap (\supp_{\preceq}\SM)^\sharp= \CJ^\sharp\cap (\supp_{\preceq}\SM)^\sharp$.  Similarly, $(\CJ^{\prime})^{\flat}\subset\CJ^\flat$ and $(\CJ^{\prime})^{\flat}\cap (\supp_{\preceq}\SM)^\sharp= \CJ^\flat\cap (\supp_{\preceq}\SM)^\sharp$. By Lemma \ref{lemma-fin} the restriction homomorphisms $\SM^{\CJ^{\sharp}}\to\SM^{(\CJ^\prime)^{\sharp}}$ and  $\SM^{\CJ^{\flat}}\to\SM^{(\CJ^\prime)^{\flat}}$ are  isomorphisms. In the diagram (\ref{diag1}) the vertical homomorphisms on the left and on the right are hence isomorphisms.  The uniqueness statement in Lemma \ref{lemma-subquotZ} then shows that also the middle vertical homomorphism is an isomorphism as well. 
 \end{proof}

The next goal is to show that $\vartheta_s\SM$ is a sheaf for any object $\SM$ in $\bC$. The following is a local version of the sheaf property.


\begin{lemma} \label{lemma-singstr} Let $\alpha\in R^+$. Suppose  that $\SM$ is a finitely supported flabby sheaf on $\CA$ with values in $\CZ^{\alpha}\catmod^f$. Suppose that $\supp_{\preceq}\SM$ is contained in  a single $\alpha$-string $\Lambda$. 
\begin{enumerate}
\item Then $\vartheta_s\SM$ is a sheaf on $\CA$ with values in $\CZ^{\alpha}\catmod^f$ and $\supp_{\preceq}\vartheta_s\SM\subset(\supp_{\preceq}\SM)^\sharp$. 
\item If $\Lambda^\sharp$ is not totally ordered, then $\vartheta_s\SM=\SN_1\oplus\SN_2$, where $\SN_1$ and $\SN_2$ are sheaves of $\CZ^{\alpha}$-modules with $\supp_{\preceq}\SN_1\subset\supp_{\preceq}\SM$ and $\supp_{\preceq}\SN_2\subset(\supp_{\preceq}\SM)s$.
\end{enumerate}
\end{lemma}

\begin{proof} For notational convenience, set $\CT:=(\supp_{\preceq}\SM)^\sharp$. First suppose that $\Lambda^\sharp=\Lambda\cup\Lambda s$ is totally ordered. Let $\CJ$ be an open subset of $\CA$,  and let $\CJ^\prime$ be the smallest open subset  containing $\CJ\cap \CT$. Then  the restriction homomorphism $(\vartheta_s\SM)^\CJ\to(\vartheta_s\SM)^{\CJ^\prime}$ is an isomorphism by Lemma \ref{lemma-firstprop}. Now $\CT$ is  finite and totally ordered. By Lemma \ref{lemma-totord}, $\vartheta_s\SM$ is hence a sheaf supported inside $\CT$. 

Now assume  that $\Lambda^\sharp$ is not totally ordered. This implies $\Lambda\ne\Lambda s$. By Lemma \ref{lemma-alphastr}, either $\Lambda$ contains only $s$-dominant elements, or only $s$-antidominant elements. We assume that $\Lambda$ contains only $s$-dominant elements. The other case is proven with analogous arguments. 

Note that since $\Lambda\ne\Lambda s$,   the  $\hCW^{\alpha}$-orbits $\Gamma_1$ through $\Lambda$ and $\Gamma_2$ through $\Lambda s$  are disjoint. By what we observed in Section \ref{subsec-candec}, there is hence a canonical decomposition $(\vartheta_s\SM)^\alpha=\SN_1\oplus\SN_2$ of presheaves of $\CZ^\alpha$-modules such that $\SN_i$ is a presheaf of $\CZ^{\alpha}(\Gamma_i)$-modules for $i=1,2$. 

 Let $\CJ$ be an open subset. We first prove the following:
 \begin{enumerate}
 \item The restriction homomorphisms $\SN_1^\CJ\to\SN_1^{\CJ^{\flat}}$ and $\SN_2^{\CJ^{\sharp}}\to\SN_2^\CJ$ are isomorphisms.
 \end{enumerate}
 As $\vartheta_s\SM$, and hence $\SN_1$ and $\SN_2$ are flabby presheaves, the restriction homomorphisms are surjective. By construction, the kernel of $(\vartheta_s\SM)^{\CJ}\to(\vartheta_s\SM)^{\CJ^\flat}$ is $\CZ$-supported inside $\pi(\CA^{-})$. As $\SN_1$ is a sheaf of $\CZ^{\alpha}(\Gamma_1)$-modules and since $\pi(\Gamma_1)\subset\pi(\CA^{+})$, the kernel of $\SN_1^{\CJ}\to\SN_1^{\CJ^{\flat}}$ is hence trivial. So the restriction is injective. Similarly, the kernel of $(\vartheta_s\SM)^{\CJ^\sharp}\to(\vartheta_s\SM)^{\CJ}$ is $\CZ$-supported inside $\pi(\CA^{+})$. As $\SN_2^{\CJ^{\flat}}$ is $\CZ$-supported on $\pi(\Gamma_2)\subset\pi(\CA^{-})$, the restriction $\SN_2^{\CJ^\sharp}\to\SN_2^\CJ$ is injective. Hence the claim is proven.
 
Now let $\CJ_1\subset\CJ$ be an  open subset with $\CJ_1\cap\supp_{\preceq}\SM=\CJ\cap\supp_{\preceq}\SM$. As $\supp_{\preceq}\SM$ contains only $s$-dominant elements, even $\CJ_1\cap(\supp_{\preceq}\SM)^\sharp=\CJ\cap(\supp_{\preceq}\SM)^\sharp$ and hence $\CJ_1^\flat\cap(\supp_{\preceq}\SM)^\sharp=\CJ^\flat\cap(\supp_{\preceq}\SM)^\sharp$. Consider now the diagram

\centerline{
\xymatrix{
\SN_1^{\CJ}\ar[r]\ar[d]&\SN_1^{\CJ^\flat}\ar[d]\\
\SN_1^{\CJ_1}\ar[r]&\SN_1^{\CJ_1^\flat}
}
}
\noindent 
(the maps are the restriction homomorphisms). The horizontals are isomorphisms by the claim above. It follows from Lemma \ref{lemma-firstprop} that the right vertical is an isomorphism (as $ \CJ_1^\flat\cap(\supp_{\preceq}\SM)^\sharp=\CJ^\flat\cap(\supp_{\preceq}\SM)^\sharp$). Hence the left vertical is an isomorphism. Now  it follows from Lemma \ref{lemma-totord} that $\SN_1$ is a sheaf supported inside $\supp_{\preceq}\SM$.

Now let $\CJ_2\subset\CJ$ be an open subset with $\CJ_2\cap(\supp_{\preceq}\SM)s=\CJ\cap(\supp_{\preceq}\SM)s$. Then $\CJ_2^\sharp\cap(\supp_{\preceq}\SM)^\sharp=\CJ^\sharp\cap(\supp_{\preceq}\SM)^\sharp$.

Consider the diagram

\centerline{
\xymatrix{
\SN_2^{\CJ^\sharp}\ar[r]\ar[d]&\SN_2^{\CJ}\ar[d]\\
\SN_2^{\CJ_2^\sharp}\ar[r]&\SN_2^{\CJ_2}
}
}
\noindent 
(the maps are the restriction homomorphisms). As before we show that  the horizontals and the right vertical are isomorphisms, hence so is the left vertical.  As before we deduce that  $\SN_2$ is a sheaf supported inside $(\supp_{\preceq}\SM)s$.\end{proof}

Now we can prove that $\vartheta_s$ indeed has all the properties listed in Theorem \ref{thm-main}.  
\begin{proposition}\label{prop-propwc} Let $\SM$ be an object in $\bC$. Then the following holds.
\begin{enumerate}
\item  $\vartheta_s\SM$ is  an object in $\bC$.
\item  For any locally closed $s$-invariant subset $\CK\subset\CA$  there is a functorial identification
$
(\vartheta_s\SM)_{[\CK]}\cong \epsilon_s(\SM_{[\CK]})
$
of $\CZ$-modules.
\item For $A\in\CA$ with $A\preceq As$ we have functorial isomorphisms
$(\vartheta_s\SM)_{[A]}\cong \SM_{[A,As]}[1]$ and $(\vartheta_s\SM)_{[As]}\cong \SM_{[A,As]}[-1]$.
\item The functor $\vartheta_s\colon\bC\to\bC$ is  exact. 
\item The functor $\vartheta_s$ preserves the subcategory $\bB$.
\item The functor $\vartheta_s\colon \bC\to\bC$ is self-adjoint.
\end{enumerate}
\end{proposition}
\begin{proof}
We first show that $\vartheta_s\SM$ is a sheaf. By Lemma \ref{lemma-locsheaf} it suffices to prove that for any $\alpha\in R^+$ the presheaf $(\vartheta_s\SM)^\alpha$ is a sheaf of $\CZ^{\alpha}$-modules on $\CA$. So fix $\alpha\in R^+$. Then $(\vartheta_s\SM)^\alpha$ can be identified with $\vartheta_s(\SM^{\alpha})$. As  $\SM$ satisfies the local extension condition, $\SM^{\alpha}$ splits into a direct sum of sheaves in such a way that each direct summand is supported inside a single $\alpha$-string. This allows us to reduce the statement to the case that $\SM$ is a sheaf of $\CZ^{\alpha}$-modules that is $\preceq$-supported inside a single $\alpha$-string. This case is settled in Lemma \ref{lemma-singstr}. This lemma also implies that $\vartheta_s\SM$ satisfies the local extension condition. By construction it is clear that $\vartheta_s\SM$ is finitely supported, its sections are objects in $\CZ\catmod^f$ and it is flabby by Lemma \ref{lemma-firstprop}. Hence in order to prove statement (1) it remains to show that $\vartheta_s\SM$ satisfies the support condition.

But we first show statement (2). Since we now know that $\vartheta_s\SM$ is a sheaf, it makes sense to talk about the subquotient $(\vartheta_s\SM)_{[\CK]}$. 
 If $\CK$ is $s$-invariant, then $\CK^-$ and $\CK^-\setminus\CK$ are $s$-invariant as well. As we can identify the  restriction homomorphism $(\vartheta_s\SM)^{\CK^-}\to(\vartheta_s\SM)^{\CK^-\setminus\CK}$ with $\epsilon_s(\SM^{\CK^-})\to\epsilon_s(\SM^{\CK^-\setminus\CK})$ we have, using the exactness of $\epsilon_s$,  $(\vartheta_s\SM)_{[\CK]}\cong\epsilon_s\SM_{[\CK]}$ functorially, hence (2).

Now we show (3).
Let $A\in\CA$ be $s$-antidominant.  As $\{A,As\}$ is a locally closed $s$-invariant subset there is a functorial identification $(\vartheta_s\SM)_{[A,As]}\cong\epsilon_s(\SM_{[A,As]})$ by (2). In the short exact sequence
$$
0\to (\vartheta_s\SM)_{[As]}\to(\vartheta_s\SM)_{[A,As]}\to(\vartheta_s\SM)_{[A]}\to 0
$$
the left hand side homomorphism is the inclusion of the maximal submodule supported inside $\pi(A)$, so this short exact sequence identifies with
$$
0\to \CZ(A,As)_{[As]}\otimes_S \SM_{[A,As]}\to \CZ(A,As)\otimes_S \SM_{[A,As]}\to \CZ(A,As)^{A}\otimes_S\SM_{[A,As]}\to 0.
$$
shifted by $[1]$. 
Claim (3) now follows from the fact that $\CZ(A,As)_{[As]}=S[-2]$ and $\CZ(A,As)^{A}\cong S$. The above argument also implies that $(\vartheta_s\SM)_{[C]}$ is $\CZ$-supported inside $\pi(C)$ for any $C\in\CA$. Hence $\vartheta_s\SM$ satisfies the support condition by Remark \ref{rem-suppcond}. Hence we finished the proof of  (1).

Now (4) and (5) follow from (3) using  Lemma \ref{lemma-ses} and Lemma \ref{lemma-Vermasubquot}, resp.

Finally, let us prove that $\vartheta_s$ is self-adjoint.
By \cite[Proposition 5.2]{FieTAMS}   the functor $\epsilon_s=\CZ\otimes_{\CZ^s}\cdot[1]$  is self-adjoint  i.e. $\Hom_{\CZ}(\epsilon_s M,N)\cong\Hom_{\CZ}(M,\epsilon_s N)$ functorially on the level of $\CZ$-modules. Let $\CJ$ be an $s$-invariant open subset of $\CA$.
As $\vartheta_s(\CX)^\CJ\cong\epsilon_s(\CX^\CJ)$ functorially for any object $\CX$ of $\bC$, there is a functorial   isomorphism 
$$
\phi_\CJ\colon\Hom_{\CZ}((\vartheta_s \SM)^{\CJ}, \SN^\CJ)=\Hom_{\CZ}(\SM^\CJ,(\vartheta_s \SN)^{\CJ})
$$
for all $\SM$ and $\SN$ in $\bC$. Moreover, for an inclusion $\CJ^\prime\subset\CJ$ of $s$-invariant subsets we can identify the diagrams

\centerline{
\xymatrix{
(\vartheta_s \SM)^{\CJ}\ar[r]\ar[d]&\SN^\CJ\ar[d]\\
(\vartheta_s \SM)^{\CJ^\prime}\ar[r]&\SN^{\CJ^\prime}
}
\text{ and }
\xymatrix{
 \SM^{\CJ}\ar[r]\ar[d]&(\vartheta_s \SN)^{\CJ}\ar[d]\\
\SM^{\CJ^\prime}\ar[r]&(\vartheta_s \SN)^{\CJ^\prime}.
}
}
\noindent 
By Proposition \ref{prop-uniquetrans}  a homomorphism $\SA\to \SB$ in $\bC$ is uniquely determined by its components $\SA^\CJ\to \SB^\CJ$ for $s$-invariant open subsets $\CJ$, and conversely, a family $\SA^\CJ\to \SB^\CJ$ of homomorphisms of $\CZ$-modules for any $s$-invariant open subset $\CJ$ that is compatible with the restriction homomorphisms, determines a morphism $\SA\to \SB$ in $\bC$. Hence the above identification yield an isomorphism
$
\Hom_{\bC}(\vartheta_s \SM,\SN)=\Hom_{\bC}(\SM, \vartheta_s \SN).
$
\end{proof}

As $\vartheta_s$ is an exact self-adjoint functor by Proposition�\ref{prop-propwc} the following is immediate.
\begin{corollary} \label{cor-transproj} Suppose that $\bP\in\bC$ is projective. Then $\vartheta_s\bP\in\bC$ is projective as well.
\end{corollary}

\subsection{Constructing  projectives via wall crossing functors}

\begin{theorem}\label{thm-projs} Let $A\in\CA$ and suppose that $A\in\Pi_\lambda$.
\begin{enumerate}
\item There is an up to isomorphism unique object $\SB(A)$ in $\bC$ with the following properties:
\begin{itemize}
\item $\SB(A)$ is indecomposable and projective in $\bC$.
\item $\SB(A)$ admits a surjective homomorphism $\SB(A)\to\SV(A)[\ell_\lambda(A)]$.
\end{itemize}
\item The object $\SB(A)$ admits a Verma flag.
\end{enumerate}
\end{theorem}

\begin{proof} 
We first show that there is at most one object (up to isomorphism) with the properties listed in (1).
Suppose that we have found an object $\SB(A)$ that has the stated properties, and suppose that $\bP$ is a projective object in $\bC$ that has an  epimorphism onto $\SV(A)[\ell_\lambda(A)]$. Then the projectivity of $\SB(A)$ and $\bP$ implies that we can find homomorphisms $f\colon \SB(A)\to\bP$ and $g\colon \bP\to\SB(A)$ such that the diagram 

\centerline{
\xymatrix{
\SB(A)\ar[dr]\ar[r]^f&\bP\ar[d]\ar[r]^g&\SB(A)\ar[dl]\\
&\SV(A)[\ell_\lambda(A)]&
}
}
\noindent 
commutes. The composition $g\circ f$ must be an automorphism, as it cannot be nilpotent and $\SB(A)$ is indecomposable (we use the Fitting decomposition in the category of sheaves on $\CA$). Hence $\SB(A)$ is a direct summand of $\bP$. If the latter is indecomposable, then $\SB(A)\cong\bP$, and hence we have proven the uniqueness statement. 
So we are left with showing that an object $\SB(A)$ in $\bC$ with the desired properties exists.

First suppose that $A=A_\lambda^-$. In this case, the object $\SB(A_\lambda^-):=\ul{\CK_\lambda}[\ell_\lambda(A_\lambda^-)]$ is contained in $\bB$, it is  indecomposable and admits an epimorphism onto $\SV(A_\lambda^-)[\ell_\lambda(A_\lambda^-)]$ by Proposition \ref{prop-ZKinC}, and it is 
projective in $\bC$ by Proposition \ref{prop-Oproj}.
Now let $A\in\Pi_\lambda$ be an arbitrary alcove. By Lemma \ref{lemma-specalc} we can then find $\lambda\in X$ and simple affine reflections $s_1$,\dots,$s_n$ in $\hCS$ with $A=A_{\lambda}^-s_1\cdots s_n$ and $A\prec A_\lambda^-s_1\cdots s_{n-1}\prec \cdots \prec A_{\lambda}^-$. As the wall crossing functors preserve projectivity,  the object $\SQ=\vartheta_{s_n} \cdots\vartheta_{s_1}\SB(A_{\lambda}^-)$ is projective in $\bC$. 
By construction, $A$ is a minimal element in $\supp_\preceq \SQ$. Hence there is an epimorphism $\SQ\to\SV(A)[\ell_\lambda(A)]$.  So we can take for $\SB(A)$ every indecomposable direct summand of $\SQ$ that maps surjectively onto $\SV(A))[\ell_\lambda(A)]$. This proves the existence part (1).  As the wall crossing functors preserve the category $\SB$, the object $\SQ$ is contained in $\SB$ and hence so is its direct summand $\SB(A)$. 
\end{proof}


\begin{thebibliography}{GKM98}

\bibitem[Fie08]{FieTAMS} \bysame, \emph{The combinatorics of {C}oxeter categories},   {T}rans.~{A}mer.~{M}ath.~{S}oc.~\textbf{360} (2008), 4211--4233.
\bibitem[Fie10]{FieDuke} Peter Fiebig, \emph{The multiplicity one case of Lusztig's conjecture}, Duke Math. J. \textbf{153} (2010), 551--571.
\bibitem[Lan12]{LanJoA} Martina Lanini, \emph{Kazhdan-Lusztig combinatorics in the moment graph setting}, J.~of Alg.~{\bf 370} (2012), 152--170.
\bibitem[Lus80]{LusAdv} George Lusztig, \emph{Hecke algebras and {J}antzen's generic decomposition  patterns}, Adv.~Math.~\textbf{37} (1980), no.~2, 121--164.

\bibitem[Soe97]{SoeRep}
Wolfgang Soergel, \emph{Kazhdan--{L}usztig-{P}olynome und eine {K}ombinatorik f\"ur {K}ipp-{M}oduln}, Represent.~Theory \textbf{1} (1997), 37--68 (electronic).
\end{thebibliography}
\end{document}